\documentclass[12pt]{article}

\usepackage{amsfonts}
\usepackage{amscd}
\usepackage{amsmath}
\usepackage{epsfig}
\usepackage{amsthm}
\usepackage{amssymb}
\usepackage{amsbsy}
\usepackage{latexsym}
\usepackage{eucal}
\usepackage{mathrsfs,bm}
\usepackage{mathtools}
\usepackage{tikz}
\usepackage[colorlinks]{hyperref}
\usepackage{color}

\DeclareMathOperator{\csch}{csch}

 \newtheorem{theorem}{Theorem}[section]

\newtheorem{proposition}[theorem]{Proposition}

\newtheorem{remark}[theorem]{Remark}
\newtheorem{lemma}[theorem]{Lemma}

\def\be{\begin{equation}}
\def\ee{\end{equation}}
\def\ben{\begin{displaymath}}
\def\een{\end{displaymath}}
\def\baa{\begin{eqnarray}}
\def\eaa{\end{eqnarray}}

\def\ba{\begin{array}}
\def\ea{\end{array}}
\renewcommand{\leq}{\leqslant}
\renewcommand{\geq}{\geqslant}

\makeatletter
\@addtoreset{equation}{section}
\makeatother

\begin{document}
\title{Spectral determinant on Euclidean isosceles triangle envelopes of fixed area as a function of angles:  absolute minimum and small-angle asymptotics}

\author{Victor Kalvin}

\date{}
\maketitle

\begin{abstract}
We study extremal properties of the determinant of Friederichs selfadoint Laplacian on the  Euclidean isosceles triangle envelopes of fixed area as a function of  angles. Small-angle asymptotics show that the determinant grows without any bound as an angle of triangle envelope goes to zero. We prove that the equilateral triangle envelope (the most symmetrical geometry) always gives rise to a critical point of the determinant and find the critical value. Moreover, if the area of envelopes is not too large, then the determinant achieves its absolute minimum only on the equilateral triangle envelope and there are no other critical points, whereas for sufficiently large area the equilateral triangle envelope corresponds  to a local maximum of the determinant.
\end{abstract}

\section{Introduction}\label{Intro}

For surfaces with  smooth varying metrics  extremal properties of determinants of Laplacians and compactness of families of isospectral metrics  were studied by  Osgood, Phillips, and Sarnak in a series of papers~\cite{OPS1,OPS2,OPS3,OPS}, see also~\cite{Sarnak} for a review of the results. 

The case of a surface without boundary  is the most simple: for all (smooth)  metrics in a given conformal class  and of given (fixed) area the determinant attains its unique absolute maximum on 
a unique constant curvature metric, called uniform,  and goes to zero as the metric degenerates~\cite{GIJR,Wo}. 

The case of a surface of type $(p,n)$, obtained by removing  $n>0$ distinct open discs from a closed surface of genus $p$,  is more involved:  for all flat (curvature zero) metrics of given (fixed) total boundary length in a given conformal class the determinant of Dirichlet Laplacian attains its unique absolute maximum on a unique flat metric, also called uniform. It is the one for which the  boundary  has constant geodesic curvature. In the case $p=0$ the determinant goes to zero as the uniform metric degenerates
~\cite{OPS3}, but this is no longer true for higher genus surfaces~\cite{Khuri2}, where for some degenerations the determinant remains bounded by a positive constant from below, or even increases without any bound. In particular, this is why the argument in~\cite{OPS2,OPS3,OPS}  allows to prove compactness of families of  isospectral metrics  in the natural $C^\infty$-topology for $np=0$ but not for $np>0$. 

The isospectral compactness of flat metrics on $(p,n)$-type surfaces with $n>0$ and negative Euler characteristic was demonstrated in~\cite{Kim}, where the total area (instead of total boundary length) of flat metrics is fixed. We also refer to~\cite{KleinKokotkin} for results on extremal properties of the determinant in the Bergman metric on the moduli space of genus two surfaces.

As discussed in~\cite{OPS}, in the case $p=0$, $n\geq 3$ the space of uniform metrics can be identified with a subset of the space $\mathscr C_n$ of conical metrics $m=c^2\prod_{j=1}^n |z-\tau_j|^{2\beta_j}|dz|^2$ on the Riemann sphere $\overline{\Bbb C}=\Bbb C\cup\{\infty\}$. Here $c$ is a scaling factor, the Gauss-Bonnet theorem requires that the orders $\beta_j>-1$ of conical singularities satisfy $\sum_{j=1}^n\beta_j=-2$, by using a suitable M\"obius transformation we always normalize  the (complex) coordinates $\tau_1,\dots,\tau_n$ of conical singularities so that $\tau_1=-1$, $\tau_2=0$, and $\tau_3=1$. 

As is well-known, the spectrum  of the Friederichs self adjoint extension of the Laplace-Beltrami operator $\Delta_m$ on $(\overline{\Bbb C},m)$ is discrete and  the corresponding zeta regularized spectral determinant $\det\Delta_m$ can be introduced in the usual way, e.g.~\cite{Khuri2}. In the same paper~\cite{Khuri2} it was shown that the determinant $\det\Delta_m$, considered as a function on the subspace $\mathscr C_n^*$ of $\mathscr C_n$ consisting of the metrics with fixed orders $\beta_j$ of conical singularities,  is real analytic.  Extremal properties of the determinant $\det\Delta_m$ on the metrics of unit area with four conical singularities of order $-1/2$  were studied in~\cite[Sec 3.5]{KokotkinPyramid}.  We also note that some variational formulas for $\det\Delta_m$ as a function on $\mathscr C^*_n$ can be obtained as an immediate consequence of results in~\cite[Proposition 1]{Kokot},~\cite{KalvinJGA}.

 An explicit expression for $\det\Delta_m$ in terms of the coordinates $c$, $\tau_j$, and $\beta_j$ on $\mathscr C_n$ was found in~\cite{Au-Sal 2} and rigorously proved in~\cite[Sec. 3.2]{Kalvin Pol-Alv}; see also~\cite{Clara-Rowlett}  for the most recent progress towards a rigorous mathematical proof of a similar explicit formula for the determinant of Dirichlet Laplacian on polygons in~\cite{Au-Sal}.

In this paper we study extremal properties of the determinant of the Friederichs selfadjoint extension of the Laplacian on the  Euclidean isosceles triangle envelopes of fixed area as a function of angles. The envelopes are glued from two congruent Euclidean isosceles triangles $ABC$ and $CBA'$ in accordance with the scheme on Fig.~\ref{Template}:
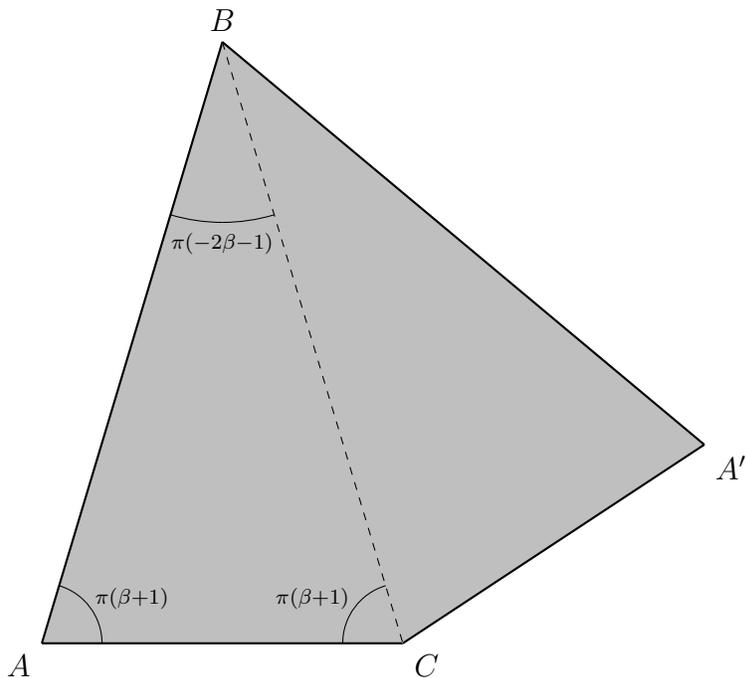
\begin{figure}[h]
\centering\begin{tikzpicture}[scale=0.8]
\draw[gray!50,fill=gray!50] (-3,-10) -- (3,-10) -- (0,0)-- (-3,-10);
\draw[black,solid,thick] (0,0)node[anchor=south]{$B$} -- (-3,-10); 
\draw[black,solid,thick] (-3,-10)node[anchor=north east]{$A$} -- (3,-10)node[anchor=north west]{$C$} ; 
\draw[gray!50,fill=gray!50,rotate=33.4] (-3,-10) -- (3,-10) -- (0,0)-- (-3,-10);
\draw[black,solid,thick,rotate=33.4] (0,0) -- (3,-10)node[anchor=north west]{$A'$}; 
\draw[black,solid,thick,rotate=33.4] (-3,-10) -- (3,-10); 
\draw[black,dashed] (0,0) -- (3,-10); 
\draw[black] (-2,-10) arc (0:73.3:1cm);
\draw[black] (2,-10) arc (180:106.7:1cm);
\draw[black] (0,-3)node[anchor=north]{$\scriptstyle\pi(-2\beta-1)$} arc (270:270-16.7:3cm);
\draw[black] (0,-3) arc (270:270+16.7:3cm);
\draw[black] (-1.5,-8.9)node[anchor=north]{$\scriptstyle\pi(\beta+1)$};
\draw[black] (1.5,-8.9)node[anchor=north]{$\scriptstyle\pi(\beta+1)$};
\end{tikzpicture}
\caption{Euclidean isosceles triangle envelope template with parameter $\beta\in(-1,-1/2)$.  How to make an envelope: $\bf A$.~Fold the polygon along the dashed line;  $\bf B$.~Glue $AC$ to $CA'$ and $AB$ to $BA'$.}
\label{Template}
\end{figure}
The polygon $ABA'C$ is first folded along the line $BC$,  then the side $AC$ is glued to $CA'$ and $AB$ is glued to $A'B$. The internal angles of the triangles are $\pi(\beta+1)$, $\pi(-2\beta-1)$, and $\pi(\beta+1)$. We use the parameter $\beta\in(-1,-1/2)$ as a measure of the angles. 

The  Euclidean isosceles triangle envelopes can equivalently be viewed as the  Riemann sphere  $\overline{\Bbb C}$ equipped with family of  flat metrics $m_\beta=c_\beta^2|z^2-1|^{2\beta}|z|^{-4-4\beta}|dz|^2$ having three  conical singularities. This family forms a subspace in $\mathscr C_3$ consisting of metrics with singularities of orders $\beta_1=\beta_3=\beta$ and $\beta_2=-2-2\beta$ and (fixed) area 
 $$S=c_\beta^2\int_{\Bbb C}|z^2-1|^{2\beta}|z|^{-4-4\beta}\frac {dz\wedge d\bar z}{-2i};$$
the scaling factor $c_\beta>0$  is uniquely determined by the area $S$ of envelope and the value of angle parameter $\beta$. The  Laplace-Beltrami operator $\Delta^S_\beta$  on the sphere $\overline{\Bbb C}$ equipped with the metric $m_\beta$ is nothing but the Euclidean Laplacian on the envelope of total area $S$ with angles prescribed by the values of $\beta$. We pick the Friederichs selfadjoint extension of $\Delta^S_\beta$, which we still denote by $\Delta^S_\beta$.

We show that for the envelopes of  fixed area $S\neq S(\beta)$ the function 
\begin{equation}\label{RAD}
(-1,-1/2)\ni\beta\mapsto\log\det\Delta^S_\beta
\end{equation}
 is real analytic and grows without any bound as $\beta\to-1^+$ or $\beta\to -1/2^-$ (i.e.  as  the envelopes degenerate, or  equivalently, as an internal angle of the triangle $ABC$ on Fig.~\ref{Template} goes to zero). We also  prove that $\beta=-2/3$, which corresponds to the equilateral triangle envelope (the most symmetrical geometry), is a critical point of the function~\eqref{RAD} and the critical value is given by 
 $$
 \log\det\Delta^S_{-2/3}=\frac 2 3 \log \pi+\frac 1 3\log\left( \frac 2 3\right) -2\log\Gamma\left(\frac 2 3 \right)+\frac 1 3 \log S.$$
 Moreover, it turns out that for all not  too large values of the area $S$ (in particular, for $S\leq 1$) the function~\eqref{RAD} achieves its \underline{absolute minimum} only at $\beta=-2/3$ (i.e. only on the equilateral triangle envelope) and there are no other critical points, whereas for sufficiently large values of $S$ (e.g. for $S\geq 2$) the critical point $\beta=-2/3$ corresponds only to a \underline{local maximum} of the function~\eqref{RAD}.  
 
   As an important step of  the proof  we  explicitly express $\log\det\Delta^S_\beta$ as a function of $S$ and $\beta$, which may be of independent interest. This is done by using an original approach based on the Meyer-Vietoris type formula for determinants of Laplacians~\cite{BFK}, formulas for the determinants of Friederichs Dirichlet Laplacians on cones~\cite{Spreafico}, integral representations for Barnes double zeta function~\cite{Spreafico2}, and the Polyakov-Alvarez anomaly formula~\cite{Alvarez,OPS1,OPS}. The Meyer-Vietoris type formula (aka Burghelea-Friedlander-Kappeler or BFK formula) for determinants of Laplacians works very much like a substitution for the insertion lemma in~\cite{OPS,Sarnak,Khuri1,Kim}. 

 We also  illustrate our purely analytical results by graphs based on explicit evaluations of the determinant for rational values of $\beta$ and uniform approximations for the function~\eqref{RAD}, cf.~Fig.~\ref{F1} and Fig.~\ref{EnvCritArea}.
 
To the best of our knowledge  no other analytical results on extremal properties of determinants under variation of angles (or, equivalently, orders of conical singularities) are available yet, except for the one in an earlier work of the author~\cite[Section 3.1]{Kalvin Pol-Alv}, where the determinant on the Riemann sphere equipped with family of singular spherical metrics is  studied as an illustrating example.

This paper is organized as follows. In the next section (Section~\ref{StR}) we state the problem in the most straightforward and naive way and formulate our main results for the Euclidean isosceles envelopes of unit area (Theorem~\ref{main}). In Section~\ref{ExplFla} we present our main technical tool: Proposition~\ref{thm1}, this is where we explicitly express the determinant  as a function of angles. In Section~\ref{Unif} we reformulate the problem in terms of flat singular metrics on the Riemann sphere. We then prove Proposition~\ref{thm1} in Section~\ref{ProofProp}. Section~\ref{Sec3} is devoted to the proof of Theorem~\ref{main}. In Section~\ref{Sec 4} we  demonstrate how the results of Theorem~\ref{main} change when considering the Euclidean isosceles triangle envelopes of non-unit area.  Finally,   Appendix~\ref{APX} contains auxiliary  results on derivatives of Barnes double zeta function that are mainly used in the proof of Theorem~\ref{main}, Section~\ref{Sec3}; derivatives of Barnes double zeta functions first appear in Proposition~\ref{thm1}.

\section{Statement of the problem and main results}\label{StR}

Consider an isosceles triangle envelope of \underline{unit area}  glued from two congruent isosceles triangles $ABC$ and $CBA'$ in accordance with the scheme on Fig.~\ref{Template}. As already discussed in Introduction~\ref{Intro},
the triangles are first glued and folded along the side $BC$,  then the side $AC$ is glued to $CA'$ and $AB$ is glued to $A'B$. The internal angles of the triangles are $\pi(\beta+1)$, $\pi(-2\beta-1)$, and $\pi(\beta+1)$. We shall use the parameter $\beta\in(-1,-1/2)$ as a measure of the angles. The pairwise glued vertices of triangles $ABC$ and $CBA'$ form the vertices  $A$, $B$, and $C$ of the triangle envelope.  

 Let $(x,y)$ be a system of Cartesian coordinates in the Euclidean plane  of the polygon $ABA'C$  on~Fig.~\ref{Template}. Clearly, this induces a  system of coordinates $(x,y)$ and a Euclidean metric on the envelope.  Let  $L^2$ stand for the space    of functions $f$ on the envelope with finite norms 
 $$
 \|f\|=\left(\iint_{ABA'C}|f(x,y)|^2\,dx\,dy\right)^{1/2};
 $$
in particular, $\|1\|=\sqrt{S}=1$. 
 
Consider the Euclidean Laplacian $\Delta_\beta=-\frac {d^2}{dx^2}-\frac {d^2}{dy^2}$ as an unbounded operator  in the Hilbert space $L^2$ initially defined on the functions $u$  that are smooth on the envelope and supported outside of its vertices $A$, $B$, and $C$. The functions $u$ can also be considered as functions in the polygon $ABA'C$. They are the smooth functions  supported outside of the vertices $A,B,A',C$, and such that the value $u(x,y)$ and the values of all derivatives $\frac {\partial^m}{\partial x^m} \frac{\partial^n}{\partial y^n} u(x,y)$ at any point $P=(x,y)$ of the side $AC$ (resp. $AB$) coincide with those at the point $P'=(x',y')$ of  $CA'$ (resp. $BA'$) satisfying $|PA|=|P'A'|$;  here $|XY|$ stands for the Euclidean distance between  $X$ and $Y$.
The points $P$ and $P'$  are getting identified after gluing $AC$ to $CA'$ and $AB$ to $BA'$.

Next we introduce the spectral zeta regularized determinant of $\Delta_\beta$. We only outline the standard  well-known steps and refer to~\cite{Khuri2} for further details (see also Section~\ref{Unif}). 

The operator $\Delta_\beta$  is densely defined, but not essentially selfadjoint.  We consider its Friederichs selfadjoint extension, which we still denote by $\Delta_\beta$ and call the Friederichs Laplacian or Laplacian for short. (The Laplacian $\Delta_\beta$ can also be viewed as an operator of a certain boundary value problem in the polygon $ABA'C$, but we do not discuss this here.)
The spectrum of $\Delta_\beta$ consists of isolated  eigenvalues $0=\lambda_0<\lambda_1\leq\lambda_2\dots$ of finite multiplicity and the determinant $\det \Delta_\beta$ is formally introduced  as their  product  with excluded eigenvalue $\lambda_0=0$. The standard zeta regularization  gives the meaning to the  infinite product:  First the corresponding spectral zeta function is  defined by the equality
 \begin{equation*}
 \zeta_\beta(s)=\sum_{j=1}^\infty\lambda_j^{-s},\quad \Re s>1;
 \end{equation*}
 asymptotics of the spectral counting function of $\Delta_\beta$ guarantees convergence of the series and analyticity of $s\mapsto \zeta_\beta(s)$ in the half plane $\Re s>1$. Then short time heat trace asymptotics are used to demonstrate that  $s\mapsto \zeta_\beta(s)$  has an analytic continuation to a neighbourhood of $s=0$. Finally, the determinant of Friederichs Laplacian $\Delta_\beta$ is rigorously defined in terms of $\zeta_\beta$ as follows: 
 $$
\det\Delta_\beta:=  \exp(-\zeta_\beta'(0)),\quad -1<\beta<-1/2.
$$

We are now in position to  formulate the main results of this paper.
\begin{theorem}[Determinant on  Euclidean  isosceles triangle envelopes of unit area]
\label{main} 
\
\begin{enumerate}

\item The function $(-1,-1/2)\ni\beta\mapsto \log\det\Delta_\beta$ is real analytic.

\item As the isosceles triangle envelope of unit area degenerates (i.e.  as an internal angle of triangle $ABC$ goes to zero, or, equivalently,  as $\beta\to -1^+$ or $\beta\to-1/2^-$), the determinant  $\det \Delta_\beta$ grows without any bound in accordance with  the following  small-angle asymptotics
\begin{equation}\label{As1}
\begin{aligned}
\log \det \Delta_\beta=&-\frac {\log(\beta+1)} {6(\beta+1)}  +\left( \frac 1 6 \log (8\pi)    -4\log A\right)\frac 1 {\beta+1}
\\
&\qquad\quad\ -\log(\beta+1)-\log 2 +O(\beta+1)\quad\text{as} \quad\beta\to-1^+,
\end{aligned}
\end{equation}
\begin{equation}\label{As2}
\begin{aligned}
&\log \det \Delta_\beta =\frac{ \log(-2\beta-1)} {12(2\beta+1)}-\left(\frac 1 6+\frac{\log\pi }{12}-2\log A\right)\frac 1{2\beta+1} 
\\
&\ \  -\frac 3 4 \log(-2\beta-1)-\frac 1 2 \log 2 -\frac 1 4 \log\pi+O(2\beta+1)\quad\text{as} \quad\beta\to-1/2^-,
\end{aligned}
\end{equation}
where $A$ is the Glaisher-Kinkelin constant.

\item The determinant $\det \Delta_\beta$ reaches its  {\underline{absolute minimum}}  only on the equilateral triangle envelope. More precisely, for $-1<\beta<-1/2$ we have 
$$
 \log\det \Delta_\beta\geq \log\det \Delta_{-2/3}=\frac 2 3 \log \pi  +\frac 1 3 \log \frac 2 3      -2\log\Gamma \left(\frac 2 3 \right) 
 $$
with equality iff $\beta=-2/3$. The function $(-1,-1/2)\ni\beta\mapsto \log\det\Delta_\beta$ is concave up, its graph is depicted on Fig.~\ref{F1}. 
\end{enumerate}
\end{theorem}

As it was already discussed in Introduction~\ref{Intro}, for the Euclidean isosceles triangle envelopes of fixed non-unit  area $S\neq S(\beta)$ the results remain essentially the same provided that the area $S$ is not too large: The equilateral triangle envelope (i.e. the most symmetrical geometry) always gives rise to a critical point of the determinant,  but the critical point turns from the \underline{absolute minimum} to a \underline{local maximum} of the determinant as the area $S$ of the envelopes increases.  We postpone a more detailed discussion to Section~\ref{Sec 4}.

\begin{figure}[h]
 \centering\includegraphics[scale=.8]{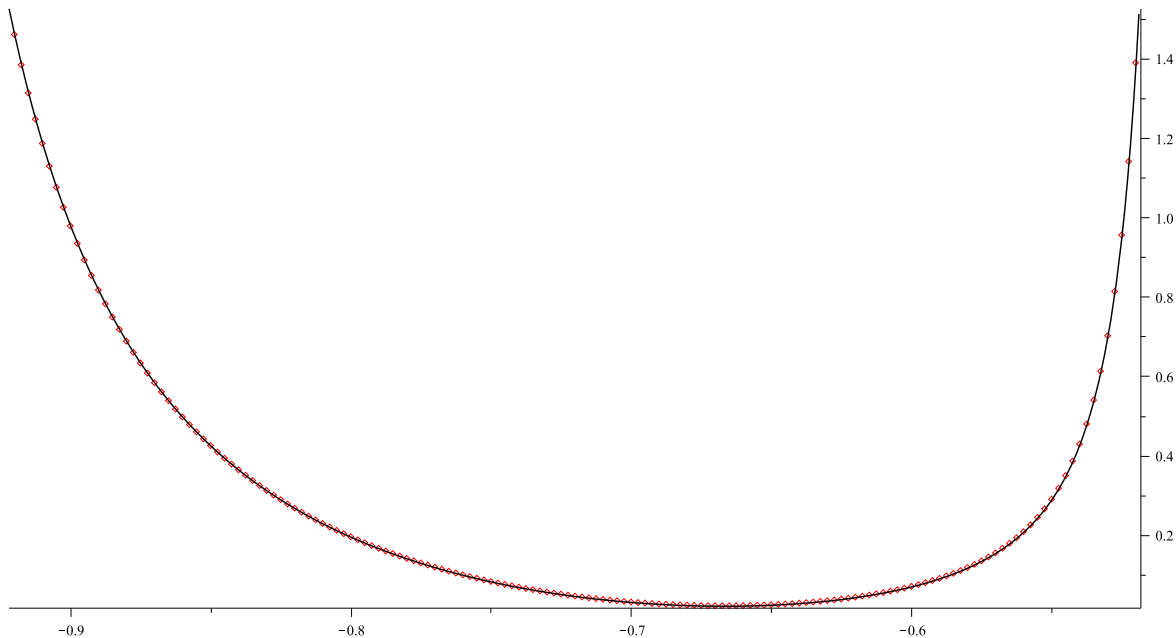}
 \caption{Graph  of the function $(-1,-1/2)\ni\beta\mapsto  \log\det \Delta_\beta$ for the  Friederichs Laplacian $\Delta_\beta$ on isosceles Euclidean triangle envelopes of unit area glued from two copies of  triangle with interior angles $\pi(\beta+1)$, $\pi(-1-2\beta)$, and $\pi(\beta+1)$, cf.~Fig~\ref{Template}.  For the graph we use the representation of $ \log\det \Delta_\beta$ as a function of $\beta$ found in Proposition~\ref{thm1} and  evaluate  $ \log\det \Delta_\beta$ for some rational values of $\beta$ by using Remark~\ref{Barnes} (the points marked with Diamond symbol on the graph), we also uniformly  approximate $ \log\det \Delta_\beta$ by using Proposition~\ref{thm1} together with estimate~\eqref{ZBE1} in Lemma~\ref{BAS}, where we take $N=4$ (solid line).}
 \label{F1}
\end{figure}

\section{Determinant as a function of angles}\label{ExplFla}
Consider the Friederichs Laplacians $\Delta_\beta$ on the Euclidean isosceles triangle envelopes of unit area.  In Proposition~\ref{thm1} below we explicitly express $\log\det \Delta_\beta$ as a function of  $\beta\in(-1,-1/2)$. 
\begin{proposition}[Determinant as a function of angles]\label{thm1}  For the determinant of the Friederichs Laplacian $\Delta_\beta$, $-1<\beta<-1/2$,  on the Euclidean isosceles triangle envelopes of unit area  we have 
\begin{equation}\label{logdet}
\begin{aligned}
\log\det \Delta_\beta=\frac 1 6  \left(2\beta+\frac 4{\beta+1}-\frac 1 {2\beta+1}\right)\log 2-\frac 1 6 \left(\frac 2 {\beta +1} -\frac 1 {2\beta+1 } -1\right)\log c_\beta
\\
-4\zeta'_B(0;\beta+1,1,1)-2\zeta'_B(0;-2\beta-1,1,1)+2\zeta_R'(-1)
\\-\log(\beta+1)-\frac 1 2 \log(-2\beta-1)-\frac {5}{2}\log 2-\log\pi.
\end{aligned}
\end{equation}
Here
 \begin{equation}\label{area}
c_\beta= \frac {2}{\Gamma\left(-\beta-1/  2\right )\Gamma(\beta+1)}  \left (\frac {\pi} {\sin\pi(-2\beta-1)}\right)^{1/2}>0
\end{equation}
is   the scaling factor that comes from uniformization of the envelopes in Section~\ref{Unif},  $\zeta_B$ stands  for the Barnes double zeta function initially defined by the double series
\begin{equation}\label{DoubleSeries}
\zeta_B(s;a,b,x)=\sum_{m,n=0}^\infty(am+bn+x)^{-s},\quad  \Re s>2, a>0, b>0, x\in\Bbb R,
\end{equation}
and then extended by analyticity to the disk $|s|<1$~\cite{Matsumoto,Spreafico2}, $\zeta_R(s)$ is the Riemann zeta function, and the prime in $\zeta'_B$  and $\zeta_R'$ denotes the derivative with respect to $s$. 
\end{proposition}

 Even though the assertion of Proposition~\ref{thm1} can be obtained as a consequence  of  much more general results~\cite[Theorem~1.1 and Proposition~3.3]{Kalvin Pol-Alv} (uniformization in Section~\ref{Unif} below is still needed), we decided to present a complete independent proof. First, because the proof of Proposition~\ref{thm1}  is much more visual, transparent, and  simple  (due to particularly simple geometry of  Euclidean isosceles triangle envelopes: spherical topology, Euclidean metrics, explicit formulas for local holomorphic and geodesic polar coordinates, only one angle parameter  $\beta$, readily available Meyer-Vietoris type formulas, etc.). Second, it makes this paper self-contained and can be of its own interest. We only note that the celebrated partially heuristic Aurell-Salomonson formula for determinants of Laplacians on polyhedra with spherical topology~\cite[(50)]{Au-Sal 2} returns a  result equivalent to the one in Proposition~\ref{thm1}; for details we refer to~\cite[Section 3.2]{Kalvin Pol-Alv}. We also refer to~\cite{Clara-Rowlett}  for the most recent progress towards a rigorous mathematical proof of a similar Aurell-Salomonson formula in~\cite{Au-Sal}, though the results and methods in~\cite{Clara-Rowlett} seem not to be directly related to ours (in the correction notice the authors refer to a pre-print which is not yet available at the time of writing this paper).

We prove  Proposition~\ref{thm1}  in  Section~\ref{ProofProp}, the proof is preceded by  Section~\ref{Unif}, Lemma~\ref{L1}, and Lemma~\ref{L2}.

\section{Uniformization}\label{Unif}
In this section we introduce a singular conformal metric $m_\beta$ on the Riemann sphere $\overline{\Bbb C}$ so that for each $\beta$ the resulting metric sphere $(\overline{\Bbb C},m_\beta)$ is isometric to the corresponding Euclidean isosceles triangle envelope glued in accordance with the scheme on Fig.~\ref{Template}. Then we reintroduce the Laplacian $\Delta_\beta$ as an unbounded selfadjoint operator in the $L^2$-space of  functions on  $\overline{\Bbb C}$. In Section~\ref{StR} this was done in an equivalent naive way, which is good for the statement of the problem and formulation of our main results, but  not suitable for the methods we use in the proof.
 
The Schwarz-Christoffel transformation
\begin{equation*}
w(z)=c_\beta \int_0^{ z}\, (\hat z^2-1)^{\beta} \hat z^{-2-2\beta}  \,d \hat z
\end{equation*}
maps the upper complex half plane $\Bbb H$ onto an isosceles triangle $ABC$ with  interior angles $\pi(\beta+1)$, $\pi(-1-2\beta)$, and  $\pi(\beta+1)$, where $-1<\beta<-1/2$. Later on the scaling factor $c_\beta>0$  will be chosen so that the triangle $ABC$ has an area of $1/2$. 
The pullback of the Euclidean metric by this conformal transformation induces the metric
\begin{equation}\label{metric}
m_\beta=c_\beta^2 |z^2-1|^{2\beta}|z|^{-4-4\beta}\,|dz|^2,\quad -1<\beta<-1/2,
 \end{equation}
  on $\Bbb H$. Similarly, the lower half plane  equipped with the same metric~\eqref{metric} is isometric to a Euclidean triangle congruent to the triangle $ABC$. Thus we obtain the following uniformization of the Euclidean isosceles triangle envelopes of unit area: the extended complex plane $\overline{\Bbb C}=\Bbb C\cup\{\infty\}$ equipped with the metric~\eqref{metric} is isometric to a  triangle envelope glued from two congruent Euclidean isosceles triangles of area $1/2$ in accordance with the scheme on Fig.~\ref{Template}. The vertex $A$ of the envelope corresponds to the point $z=-1$, the vertex $B$ corresponds to $z=0$, and $C$ to $z=1$. The points $0$ and $\pm1$ of $\overline{\Bbb C}$ will also be called vertices.

It is natural to pick the local complex coordinates $z\in\Bbb C$ and $1/z$ near $\infty$ on $\overline{\Bbb C}$. In the  coordinate $z$ the Laplacian $\Delta_\beta$   takes the form
$$
\Delta_\beta= -4 c_\beta^{-2}|z^2-1|^{-2\beta}|z|^{4+4\beta}\partial_{z}\partial_{\bar z},
$$
where $\partial_{z}=\frac{\partial}{\partial z}$ and  $\partial_{\bar z}=\frac{\partial}{\partial\bar z} $ are the complex  derivatives. The change of coordinates  $z\mapsto1/z$ leads to the corresponding representations of the metric $m_\beta$ and the  Laplacian $\Delta_\beta$ in the  local coordinate near $\infty$. 

In a vicinity of the vertex $z=\pm 1$ there is a local complex isothermal coordinate $w_{\pm}$ that brings the metric~\eqref{metric} into the form $m_\beta=|w_\pm|^{2\beta}|dw_\pm|^2$ for $|w_\pm|<\delta$ with some $\delta>0$, see e.g.~\cite{TroyanovSSC},\cite[Lemma 3.4]{Troyanov Polar Coordinates} or~\cite[Section 1]{OPS}. For the holomorphic change of coordinate $w_\pm=w_\pm(z)$ we have
\begin{equation}\label{w_pm}
w_\pm(z)=\left((\beta+1)c_\beta\int_{\pm 1}^z   (\hat z^2-1)^{\beta}\hat z^{-2-2\beta}\,d\hat z\right)^{\frac 1 {\beta+1}},
\end{equation}
where $z$ is the same as in~\eqref{metric}, the sign $+$ (resp. $-$) is taken for the vertex $z=1$ (resp. $z=-1$).  Let us also notice that in the geodesic polar coordinates 
$$
(r,\theta)=\left((\beta+1)^{-1}|w_\pm|^{\beta+1}, \arg w_\pm \right)
$$
 near the vertex $z=\pm1$ the metric takes the form 
$$
m_\beta=d r^2+(1+\beta)^2r^2\, d\theta^2.
$$
The latter one is  the  metric of a standard cone  (surface of revolution) created by rotating an angle of size  $\alpha\in(0,\pi/6)$, $\sin\alpha=\beta+1$,  around one of its rays. 
Similarly, in the local isothermal coordinate
\begin{equation}\label{w_0}
w_0=w_0(z)=\left(- (2\beta+1)c_\beta\int_0^z   (\hat z^2-1)^{\beta}\hat z^{-2-2\beta}\,d\hat z\right)^{-\frac 1 {2\beta+1}}
\end{equation}
in a vicinity of the vertex $z=0$ we have $m_\beta=|w_0|^{-4-4\beta}\,|dw_0|^2$  for $|w_0|<\delta$ with some $\delta>0$.  In the geodesic polar coordinates $$
(r,\theta)=\left((-2\beta-1)^{-1}|w_0|^{-2\beta-1}, \arg w_0 \right)
$$near $z=0$ the metric $m_\beta$ takes the from of a standard cone of vertex angle $\alpha\in(0,\pi/2)$ such that  $\sin\alpha=-1-2\beta$.

The space  $L^2$, considered as a space of functions on $\overline{\Bbb C}$,  has the norm
\begin{equation}\label{L2norm}
\|f\|=c_\beta\left( \int_{\Bbb C}  |z^2-1|^{2\beta}|z|^{-4-4\beta} |f(z,\bar z)|^2  \frac {dz\wedge d\bar z}{-2i}  \right)^{1/2}.
\end{equation}
 The operator $\Delta_\beta$  in the  Hilbert space $L^2$ is initially defined on the smooth functions supported outside of the vertices $z=0$ and $z=\pm1$. This operator is  densely defined but  \underline{not} essentially selfadjoint (due to the singularities at $z=0$ and $z=\pm1$). We consider the Friederichs selfadjoint extension of $\Delta_\beta$, which we still denote by $\Delta_\beta$. Recall that it is the only selfadjoint extension with domain in  $H^1$~\cite{Kato}; here $H^1$ stands for the Sobolev space of all functions $u\in L^2 $ with finite Dirichlet integral  $\|\nabla_\beta u\|^2$, where $\nabla_\beta u$ is the gradient of a function $u$ with respect to the metric~$m_\beta$ and $\|\nabla_\beta u\|^2$ is the integral in~\eqref{L2norm} with $f=\nabla_\beta u$. The well-known definition of the zeta regularized spectral determinant $\det\Delta_\beta$ was already discussed in Section~\ref{StR}; for more details, we refer to~\cite{Khuri2}.
 
 Finally, in order to find the announced in Proposition~\ref{thm1} value of the scaling factor $c_\beta$,  we first note that the length  of the side $AB$ is the metric distance  between the vertices $z=-1$ and $z=0$ (or, what is the same, the length of $BC$ is the metric distance between $z=0$ and $z=1$). The length of $AB$ is given by
$$
|AB|=|BC|=  c_\beta\int_0^1|z^2-1|^{\beta}|z|^{-2-2\beta}\,|dz|= c_\beta \frac {\Gamma\left(-\beta-\frac 1 2\right)\Gamma(\beta+1)}{2\sqrt{\pi}}.
$$
For the total area of the envelope we  thus have
$$
\begin{aligned}
1=2|AB|\cdot|BC|\cdot\sin\frac \pi 2 (-1-2\beta)\cdot\cos\frac \pi 2 (-1-2\beta)\\=  \frac 1 {4\pi}c_\beta^2\Bigl (\Gamma\left(-\beta-1/  2\right )\Gamma(\beta+1)\Bigr)^2\cdot\sin\pi(-2\beta-1),
\end{aligned}
$$
this implies~\eqref{area}.

\section{Proof of Proposition~\ref{thm1}}\label{ProofProp}

 We  need some preliminaries before we can formulate the first lemma preceding the proof of Proposition~\ref{thm1}.

Let $\overline{ \Bbb C}_\epsilon$   (respectively  $\Bbb C_\epsilon$)  stand for the Riemann sphere $\overline{\Bbb C}$  (respectively the complex  plane $\Bbb C$) with three small disks $|w_-|<\epsilon$, $|w_0|< \epsilon$, and $|w_+|< \epsilon$  removed. Here $w_\pm$ and $w_0$ are the local complex isothermal coordinates~\eqref{w_pm} and~\eqref{w_0} in vicinities of the vertices. 
Note that the construction of $\overline{ \Bbb C}_\epsilon$ is nearly identical to the one for $\Sigma_3$ in~\cite[Section 1]{OPS}.  

As it was discussed in Section~\ref{Unif}, any Euclidean isosceles triangle envelope of unit area can  equivalently be viewed as the extended complex plane $\overline {\Bbb C}=\Bbb C\cup\{\infty\}$  (Riemann sphere) endowed with the metric $m_\beta$ in~\eqref{metric}. Notice that
the surface $(\overline{ \Bbb C}_\epsilon,m_\beta)$ with boundary $\partial\Bbb C_\epsilon$ (a  flat  pair of  pants) is isometric to the Euclidean isosceles triangle envelope of unit area with $\epsilon$-cones around the vertices removed (i.e. the surface can be glued from two congruent Euclidean isosceles triangles with excision of circular sectors of radius $\frac 1{\beta+1}\epsilon^{\beta+1}$ at the vertices $A$, $A'$ and $C$, and of radius $-\frac 1{2\beta+1}\epsilon^{-2\beta-1}$ at the vertex $B$, see Fig.~\ref{Template} and Fig.~\ref{Template2}).

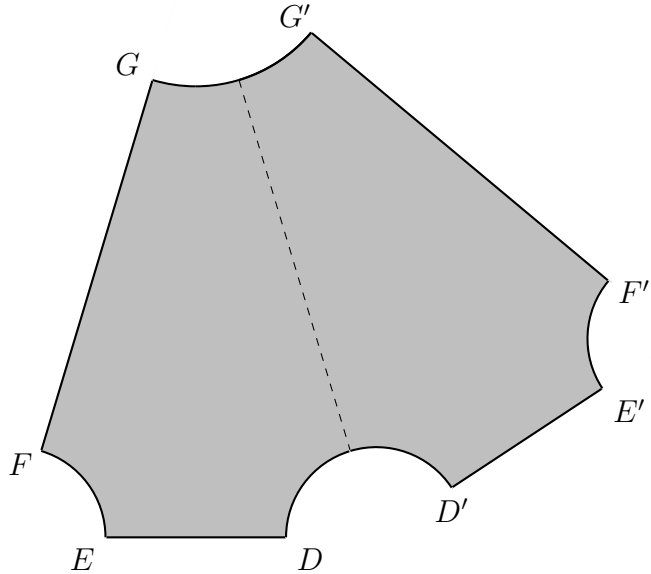
\begin{figure}[h]
\centering\begin{tikzpicture}[scale=0.8]
\draw[gray!50,fill=gray!50] (-3,-10) -- (3,-10) -- (0,0)-- (-3,-10);
\draw[black,solid,thick,rotate=-16.7] (0,-2.5) -- (0,-9); 
\draw[black,solid,thick] (-1.5,-10) -- (1.5,-10); 
\draw[gray!50,fill=gray!50,rotate=33.4] (-3,-10) -- (3,-10) -- (0,0)-- (-3,-10);
\draw[black,solid,thick,rotate=3*16.7] (0,-2.5) -- (0,-9); 
\draw[black,solid,thick,rotate=33.4] (-1.5,-10) -- (1.5,-10); 
\draw[black,dashed,rotate=16.7] (0,-2.5) -- (0,-9); 
\draw[white,fill=white](-3,-10) circle (1.5cm);
\draw[white,fill=white,thick] (-3,-10)--(-1.5,-10) arc (0:73.3:1.5cm)--cycle;
\draw[white,fill=white,thick]  (3,-10)--(1.5,-10) arc (180:106.7:1.5cm)--cycle;
\draw[white,fill=white,thick] (0,0)--(0,-2.5)arc (270:270-16.7:2.5cm)--cycle;
\draw[white,fill=white,thick] (0,0)--(0,-2.5) arc (270:270+16.7:2.5cm)--cycle;
\draw[black,thick] (-1.5,-10) arc (0:73.3:1.5cm);
\draw[black,thick] (1.5,-10) arc (180:106.7:1.5cm);
\draw[black,thick] (0,-2.5)arc (270:270-16.7:2.5cm);

\draw[white,fill=white,thick,rotate=33.4] (-3,-10)--(-1.5,-10) arc (0:73.3:1.5cm)--cycle;
\draw[white,fill=white,thick,rotate=33.4]  (3,-10)--(1.5,-10) arc (180:106.7:1.5cm)--cycle;
\draw[white,fill=white,thick,rotate=33.4] (0,0)--(0,-2.5)arc (270:270-16.7:2.5cm)--cycle;
\draw[white,fill=white,thick,rotate=33.4] (0,0)--(0,-2.5) arc (270:270+16.7:2.5cm)--cycle;
\draw[black,thick] (0,-2.5) arc (270:270+3*16.7:2.5cm);
\draw[black,thick,rotate=33.4] (-1.5,-10) arc (0:73.3:1.5cm);
\draw[black,thick,rotate=33.4] (1.5,-10) arc (180:106.7:1.5cm);
\draw[black,thick,rotate=33.4] (0,-2.5)arc (270:270-16.7:2.5cm);
\draw[black,thick,rotate=33.4] (0,-2.5) arc (270:270+16.7:2.5cm);
\draw[black] (-1.5,-10)node[anchor=north east]{$E$};
\draw[black] (1.5,-10)node[anchor=north west]{$D$};
\draw[black,rotate=33.4] (-1.5,-10)node[anchor=north]{$D'$};
\draw[black,rotate=33.4] (1.5,-10)node[anchor=north west]{$E'$};
\draw[black,rotate=-16.7] (0,-8.8)node[anchor=north east]{$F$};
\draw[black,rotate=3*17] (0,-8.8)node[anchor=north west]{$F'$};
\draw[black,rotate=-16.7] (0,-2.6)node[anchor=south east]{$G$};
\draw[black,rotate=3*17] (0,-2.7)node[anchor=south east]{$G'$};
\end{tikzpicture}
\caption{ How to make a flat pair of pants $(\overline{\Bbb C}_\epsilon,m_\beta)$: $\bf A$.~Cut out four  circular sectors (of radius $\frac 1{\beta+1}\epsilon^{\beta+1}$ at the vertices $A$, $A'$ and $C$, and of radius $-\frac 1{2\beta+1}\epsilon^{-2\beta-1}$ at the vertex $B$) from the Euclidean isosceles triangle envelope template on Fig.~\ref{Template} as shown here; $\bf B$.~Fold along the dashed line;  $\bf C$.~Glue $DE$ to $D'E'$ and $FG$ to $F'G'$.}
\label{Template2}
\end{figure}

By $\Delta_\beta\!\!\restriction_{\overline{ \Bbb C}_\epsilon}$ we denote the selfadjoint Dirichlet Laplacian on $(\overline{ \Bbb C}_\epsilon,m_\beta)$ (since there are no singularities of $m_\beta$ on $\overline{ \Bbb C}_\epsilon$, the Dirichlet Laplacian $\Delta_\beta\!\!\restriction_{\overline{ \Bbb C}_\epsilon}$  initially defined on the smooth functions is  essentially selfadjoint). The operator $\Delta_\beta\!\!\restriction_{\overline{ \Bbb C}_\epsilon}$  is positive and its spectrum consists of discrete eigenvalues of finite multiplicity. Let $\det \Delta_\beta\!\!\restriction_{\overline{ \Bbb C}_\epsilon}$ stand for the zeta regularized spectral determinant of $\Delta_\beta\!\!\restriction_{\overline{ \Bbb C}_\epsilon}$. 

We will also be using the standard spherical metric 
\begin{equation}\label{roundsphere}
m=4(1+|z|^2)^{-2}|dz|^2
\end{equation}
as a reference metric on  $\overline{\Bbb C}$. The surface $(\overline{\Bbb C},m)$ is isometric to a unit sphere in $\Bbb R^3$  (and  $(\overline{ \Bbb C}_\epsilon,m)$ is isometric to a unit sphere  in $\Bbb R^3$ with three holes).   Similarly to the notation for the determinant  generated by the metric $m_\beta$ on $\overline{ \Bbb C}_\epsilon$  introduced above,  we denote by    $\det{\Delta}\!\!\restriction_{\overline{ \Bbb C}_\epsilon}$   the determinant of  the selfadjoint Dirichlet Laplacian ${\Delta}\!\!\restriction_{\overline{ \Bbb C}_\epsilon}$ on $(\overline{ \Bbb C}_\epsilon,m)$.

\begin{lemma}\label{L1} For $-1<\beta<-1/2$ the determinants of the Dirichlet Laplacians   $\Delta_\beta\!\!\restriction_{\overline{ \Bbb C}_\epsilon}$ and ${\Delta}\!\!\restriction_{\overline{ \Bbb C}_\epsilon}$   satisfy
 $$
 \begin{aligned}
\log\frac{ \det (\Delta_\beta\!\!\restriction_{\overline{ \Bbb C}_\epsilon})}{\det ({\Delta}\!\!\restriction_{\overline{ \Bbb C}_\epsilon})}= &\frac {3\beta^2+4\beta} 3 \log \epsilon +\frac {\beta^2+3\beta+1}{3(\beta+1)}\log 2  
\\
&+\frac 1 {6} \left(\frac {2} {\beta+1} -\frac 1 {2\beta+1}+1\right)\log c_\beta-\frac 4 3+o(1)\quad\text{as} \quad \epsilon\to0^+,
\end{aligned}
$$
where $c_\beta$ is the scaling factor found in~\eqref{area}. 
\end{lemma}
\begin{proof}[Proof of Lemma~\ref{L1}] Introduce the metric  potentials 
\begin{equation}\label{chi}
\chi_\beta(z)=  \log c_\beta +\beta\log |z+1|+\beta\log |z-1|-2(1+\beta)\log |z|,
\end{equation}
\begin{equation}\label{psi}
\psi(z)= \log 2 -\log(1+|z|^2),
\end{equation}
and $\varphi_\beta= \chi_\beta-\psi$. In terms of these potentials we have $m_\beta=e^{2\chi_\beta}|dz|^2$ for the metric~\eqref{metric}, $m=e^{2\psi}|dz|^2$ for the metric~\eqref{roundsphere}, and $m_\beta=e^{2\varphi_\beta}m$. The potential $\psi$ is smooth on $\overline{\Bbb C}$, while $\chi_\beta$ and $\varphi_\beta$ are smooth on $\overline{ \Bbb C}_\epsilon$ (but not on $\overline{\Bbb C}$).

The well-known Polyakov-Alvarez  formula~\cite{Alvarez, OPS1, OPS} implies
\begin{equation}\label{P_A}
\begin{aligned}
\log\frac{ \det (\Delta_\beta\!\!\restriction_{\overline{ \Bbb C}_\epsilon})}{\det ({\Delta}\!\!\restriction_{\overline{ \Bbb C}_\epsilon})}=-\frac{1}{12\pi}\int_{\Bbb C_\epsilon} (|\nabla\varphi_\beta|^2&+2K\varphi_\beta)\,dA
\\
&-\frac{1}{6\pi}\int_{\partial \Bbb C_\epsilon}k\varphi_\beta \,d s -\frac 1 {4\pi}\int_{\partial \Bbb C_\epsilon}{\partial_{ n}}\varphi_\beta\,d s.
\end{aligned}
\end{equation}
Here   $\nabla$ is the gradient, $K=1$ is the Gaussian curvature, $dA$ is the area,  $k$ is the geodesic curvature of the boundary $\partial\Bbb C_\epsilon$,  $\partial_n$ is the outward normal derivative to $\partial\Bbb C_\epsilon$, and $ ds$ is the arc length --- all with respect to  the spherical metric $m$.   

The Liouville's equation for the curvature $K_\beta=0$ of the metric $m_\beta$ on $\Bbb C_\epsilon$ reads $e^{-2\chi_\beta}(-4\partial_z\partial_{\bar z }\chi_\beta)=K_\beta$. Hence $\chi_\beta$ is a harmonic function on $\Bbb C_\epsilon$ and 
\begin{equation}\label{LAP}
-\Delta \varphi_\beta=-e^{-2\psi} (-4\partial_z\partial_{\bar z }\varphi_\beta)=e^{-2\psi} (-4\partial_z\partial_{\bar z }\psi)=1   
\end{equation}
is the curvature of $m$ on  $\Bbb C_\epsilon$.
Now we use Green's formula together with~\eqref{LAP} and write~\eqref{P_A} in the form
\begin{equation}\label{PA}
\begin{aligned}
\log\frac{ \det (\Delta_\beta\!\!\restriction_{\overline{ \Bbb C}_\epsilon})}{\det ({\Delta}\!\!\restriction_{\overline{ \Bbb C}_\epsilon})}=-\frac{1}{12\pi}\left (\int_{\Bbb C_\epsilon} \varphi_\beta\,dA\right. &\left.+ \int_{\partial \Bbb C_\epsilon} \varphi_\beta (\partial_n\varphi_\beta )\,ds\right)
\\
&-\frac {1}{6\pi}\int_{\partial \Bbb C_\epsilon}k\varphi_\beta \,d s -\frac 1 {4\pi}\int_{\partial \Bbb C_\epsilon}\partial_n\varphi_\beta\,d s. 
\end{aligned}
\end{equation}

For the area integral in~\eqref{PA} we have
\begin{equation*}\label{areaint}
\lim_{\epsilon\to 0^+}\int_{\Bbb C_\epsilon} \varphi_\beta\,dA=\int_{\Bbb C} \chi_\beta\,dA-\int_{\Bbb C} \psi\,dA=4\pi\left((\beta-1)\log 2+\log c_\beta+1\right).
\end{equation*}
Indeed, based on~\eqref{chi} and~\eqref{psi} both integrals over $\Bbb C$  can be easily  evaluated  as follows: 
$$
\begin{aligned}
\int_{\Bbb C} \chi_\beta\,dA &
 =\int_{\Bbb C} \chi_\beta (-4\partial_z\partial_{\bar z } \psi) \frac {dz\wedge d\bar z}{-2i} 
\\
&=\lim_{\delta\to 0^+}\left(\oint_{|z+1|=\delta}+\oint_{|z|=\delta} +\oint_{|z-1|=\delta}+\oint_{|z|=1/\delta} \right) (\chi_\beta \partial_{ n}\psi-\psi\partial_{ n}\chi_\beta) |dz|
\\
& =-2\pi \beta\psi(-1) -2\pi(-2-2\beta)\psi(0) -2\pi \beta\psi(1) -2\pi(2\log 2 -2\log c_\beta)  \\&=4\pi\left(\beta  \log 2+\log c_\beta\right),
\\
\int_{\Bbb C} \psi\,dA&=\int_{\Bbb C}\bigl(\log 2 -\log(1+|z|^2)\bigr) \frac {4}{(1+|z|^2)^2}\frac {dz\wedge d\bar z}{-2i} =4\pi (\log 2- 1). 
   \end{aligned}
$$

Next we find the asymptotics of  all other integrals in~\eqref{PA} as $\epsilon\to 0^+$. Since the equations of all three components of the boundary $\partial \Bbb C_\epsilon$ are  written in different local coordinates, it is convenient to treat  the components  in the corresponding local coordinates separately. The computations for all three components are very similar,  we present detailed evaluations only  for the components $|w_\pm|=\epsilon$ encircling the vertices $z=\pm1$.  

Clearly, the geodesic curvature of the circle $|w_\pm|=\epsilon$ with respect to the usual Euclidean metric $|dw_\pm|^2$ is $\epsilon^{-1}$. In order to find the geodesic curvature $k$ with respect to the spherical metric $m=e^{2\psi_\pm}|dw_\pm|^2$  we use the equality 
$$k=-e^{-\psi_\pm}(\epsilon^{-1}+\partial_{|w_\pm|}\psi_\pm),$$
where $\psi_\pm$ is the (smooth) potential of the spherical metric $m$ written in the local holomorphic coordinate $w_\pm$. For the arc length we have $d s= e^{\psi_\pm} \,|d w_\pm|$. We also recall that  $m_\beta=|w_\pm|^{2\beta}|dw_\pm|^2$ and  $\partial_{ n}= e^{-\psi_\pm}\partial_{|w_\pm|}$. Thus for the components of $\partial\Bbb C_\epsilon$  defined by the equations $|w_\pm|=\epsilon$ we obtain
$$
\begin{aligned}
\oint_{|w_\pm|=\epsilon} \varphi_\beta \partial_{ n}\varphi_\beta \,ds
= \oint_{|w_\pm|=\epsilon} \bigl (\beta \log|w_\pm|-\psi_\pm(w_\pm)\bigr)\partial_{|w_\pm|}\bigl(\beta\log|w_\pm|-\psi_\pm(w_\pm)\bigr)\,|dw_\pm|
\\
 =-\int_0^{2\pi}\bigl(\beta\log \epsilon-\psi_\pm(\epsilon e^{i\theta})\bigr)
\left({\beta}/\epsilon+O(1)\right)\epsilon\, d\theta =-2\pi\beta^2\log\epsilon+2\pi\beta\psi_\pm(0)+o(1),
\end{aligned}
$$
$$
\begin{aligned}
\oint_{|w_\pm|=\epsilon}  & k\varphi_\beta \,d s =-\oint_{|w_\pm|=\epsilon}  \bigl(\epsilon^{-1}+\partial_{|w_\pm|}\psi_\pm(w_\pm)\bigr)\bigl(\beta \log \epsilon -\psi_\pm(w_\pm)\bigr)\,|dw_\pm|
\\
& =-2\pi\beta\log\epsilon+2\pi\psi_\pm(0)+o(1),
\end{aligned}
$$
$$
\begin{aligned}
\oint_{|w_\pm|=\epsilon}\partial_{ n}\varphi_\beta \,ds=& \oint_{|w_\pm|=\epsilon} \partial_{|w_\pm|}\bigl(\beta \log|w_\pm|-\psi_\pm(w_\pm)\bigr)\,|dw_\pm|
=-2\pi\beta+o(1).
\end{aligned}
$$
These estimates would be useless without knowing the value of $\psi_\pm$ at zero. Fortunately, it is not hard to show that 
\begin{equation}\label{psi_pm_0}
\psi_\pm(0)=\psi(\pm1)-\log |w'_\pm(\pm1)|=-\frac \beta{\beta+1}\log 2 -\frac 1 {\beta+1}\log c_\beta
\end{equation}
with the metric potential $\psi$ given in~\eqref{psi} and  the function $w_\pm(z)$  given in~\eqref{w_pm}.

In exactly the same way we also find that
$$
\begin{aligned}
\oint_{|w_0|=\epsilon} \varphi_\beta \partial_n\varphi_\beta \,ds
=-8\pi(\beta+1)^2\log\epsilon-4\pi(\beta+1)\psi_0(0)+o(1),
\end{aligned}
$$
$$
\oint_{|w_0|=\epsilon}   k\varphi_\beta \,d s =4\pi(\beta+1)\log\epsilon+2\pi\psi_0(0)+o(1),\quad 
\oint_{|w_0|=\epsilon}\partial_n\varphi_\beta \,ds=4\pi(\beta+1)+o(1),
$$
where for the potential $\psi_0(w_0)$ of the spherical metric $m$ written in the local holomorphic coordinate~$w_0=w_0(z)$ in~\eqref{w_0}  we have
\begin{equation}\label{psi_0_0}
\psi_0(0)=\psi(0) -\log|w_0'(0)|= \log 2 +\frac 1{2\beta+1}\log c_\beta. 
\end{equation}

We have evaluated all integrals in the right hand side of~\eqref{PA}. Combining the results we obtain
 $$
 \begin{aligned}
\log\frac{ \det (\Delta_\beta\!\!\restriction_{\overline{ \Bbb C}_\epsilon})}{\det ({\Delta}\!\!\restriction_{\overline{ \Bbb C}_\epsilon})}= &-\frac {1} 3 \bigl((\beta-1)\log 2+ \log c_\beta+1\bigr)
\\
&+\left(\frac {\beta (\beta+2)} 3\log \epsilon -\frac {\beta+2}3\psi_\pm(0)+\beta\right)
\\
&+\left(\frac {2\beta(\beta+1)} 3 \log\epsilon+\frac {\beta}3\psi_0(0)- (\beta+1)\right)+o(1)\quad\text{as}\quad\epsilon\to0^+
\end{aligned}
$$
with $\psi_\pm(0)$ and $\psi_0(0)$ given in~\eqref{psi_pm_0} and~\eqref{psi_0_0}.  This completes the proof of Lemma~\ref{L1}. 
\end{proof}

In the next lemma  we study the determinants of  Dirichlet Laplacians on the $\epsilon$-cones and spherical $\epsilon$-disks cut out of the triangle envelope $(\overline{\Bbb C},m_\beta)$ and the curvature one sphere $(\overline{\Bbb C},m)$ respectively.
Since the metric $m_\beta$ of the $\epsilon$-cones 
$$
\left(|w_\pm|\leq\epsilon,|w_\pm|^{2\beta}|dw_\pm|^2\right)\quad\text{and}\quad \left(|w_0|\leq \epsilon,|w_0|^{-4-4\beta}|dw_0|^2\right)
$$ 
is singular, the corresponding Dirichlet Laplacians $\Delta_\beta\!\!\restriction_{|w_\pm|\leq \epsilon}$ and $\Delta_\beta\!\!\restriction_{|w_0|\leq \epsilon}$, initially defined on a dense set of smooth functions supported outside of the vertices, are not essentially selfadjoint. As before,  we pick the Friederichs selfadjoint extension. By $\Delta\!\!\restriction_{|w_\pm|\leq \epsilon}$  (resp.  $\Delta\!\!\restriction_{|w_0|\leq \epsilon}$) we denote the selfadjoint Dirichlet Laplacians on the disks $|w_\pm|\leq \epsilon$ (resp. $|w_0|\leq \epsilon$)  endowed with the spherical metric $m$. 

\begin{lemma}\label{L2}  Let $-1<\beta<-1/2$. Then 
\begin{enumerate} \item The determinants of Dirichlet Laplacians $\Delta_\beta\!\!\restriction_{|w_\pm|\leq \epsilon}$ and ${\Delta}\!\!\restriction_{|w_\pm|\leq \epsilon}$   satisfy
$$
\begin{aligned}
&\log\frac{ \det (\Delta_\beta\!\!\restriction_{|w_\pm|\leq \epsilon})}{\det ({\Delta}\!\!\restriction_{|w_\pm|\leq \epsilon})}=-\frac { \beta^2+2\beta} 6\log\epsilon +\frac 1 6\frac {\beta^2-2\beta} {\beta+1}\log 2 -\frac 1 {3(\beta+1)}\log c_\beta
\\
&\ -\frac 5 {12}\beta -\frac 1 2 \log(\beta+1)-2\zeta_B'(0;\beta+1,1,1)
+2\zeta'_R(-1) +o(1)\quad \text{as} \quad \epsilon\to0^+.
\end{aligned}
$$
 \item The determinants of Dirichlet Laplacians $\Delta_\beta\!\!\restriction_{|w_0|\leq \epsilon}$ and ${\Delta}\!\!\restriction_{|w_0|\leq \epsilon}$  satisfy
$$
\begin{aligned}
&\log\frac{ \det (\Delta_\beta\!\!\restriction_{|w_0|\leq \epsilon})}{\det ({\Delta}\!\!\restriction_{|w_0|\leq \epsilon})}= -\frac {2 \beta^2+2\beta} 3 \log\epsilon-\frac {2\beta^2+2\beta+1} {3(2\beta+1)}\log 2+\frac 1{3(2\beta+1)}\log c_\beta
\\ 
&+\frac 5 {6}(\beta+1) -\frac 1 2 \log(-2\beta-1) -2\zeta_B'(0;-2\beta-1,1,1)+2\zeta'_R(-1)+o(1)\   \text{as} \ \epsilon\to0^+.
\end{aligned}
$$
\end{enumerate}

\end{lemma}
\begin{proof}[Proof of Lemma~\ref{L2}]  The Polyakov-Alvarez formula~\cite{Alvarez, OPS1,OPS} for  Dirichlet Laplacians in two  metrics on the disk $|w_\pm|\leq \epsilon$ reads:  
\begin{equation}\label{aux}
\begin{aligned}
\log \frac {\det(\Delta\!\!\restriction_{|w_\pm|\leq \epsilon})}{\det(\Delta_\flat\!\!\restriction_{|w_\pm|\leq \epsilon})}= -\frac{1}{6\pi}\left (\frac1 2 \int_{|x|\leq\epsilon} |\nabla_\flat\psi_\pm|^2\,\frac{dw_\pm\wedge d \bar w_\pm}{-2i}+\oint_{|w_\pm|=\epsilon}k_\flat\psi_\pm \,|d w_\pm|  \,\right )
\\
-\frac 1 {4\pi}\oint_{|w_\pm|=\epsilon}(\partial_{n_\flat}\psi_\pm)\,|d w_\pm|.
\end{aligned}
\end{equation}
Here $\psi_\pm$ stands for the potential of the spherical metric $m=e^{2\psi_\pm}|dw_\pm|^2$ written in the local coordinate $w_\pm=w_\pm(z)$ in~\eqref{w_pm} and the symbol $\flat$ refers to the flat metric $|dw_\pm|^2$ in the disk $|w_\pm|\leq \epsilon$. Thus $\nabla_\flat$ is the gradient, $k_\flat=1/\epsilon$ is the geodesic curvature of the circle $|w_\pm|=\epsilon$,  $\partial_{n_\flat}$ is the outer normal derivative,  and $\Delta_\flat\!\!\restriction_{|w_\pm|\leq \epsilon}$ is the selfadjoint Dirichlet Laplacian.   
As is well-known,
$$
\log \det(\Delta_\flat\!\!\restriction_{|w_\pm|\leq \epsilon})=-\frac 1  3 \log\epsilon +\frac 1 3 \log 2-\frac 1 2 \log( 2\pi) -\frac 5 {12} -2\zeta'_R(-1);
$$
 see~e.g.~\cite[first equality in (28)]{Weisberger} or~\cite[Corollary 1]{Spreafico}. 
We note that in the Polyakov-Alvarez formula~\eqref{aux} only the integral involving the geodesic curvature $k_\flat$ gives a nonzero input of $-\frac 1 3 \psi_\pm(0)+o(1)$ as $\epsilon\to0^+$, while all other integrals go to zero.  Taking into account~\eqref{psi_pm_0}  we thus obtain
\begin{equation}\label{sst}
\begin{aligned}
\log\det (\Delta\!\!\restriction_{|w_\pm|\leq \epsilon})=-\frac 1 3\log\epsilon &+ \frac  {2\beta+1}{3(\beta+1)} \log2  +\frac 1 {3(\beta+1)}\log c_\beta
\\
 &-2\zeta'_R(-1) -\frac 5 {12}  -\frac 1 2 \log (2\pi) +o(1) \quad\text{as } \epsilon\to 0^+. 
\end{aligned}
\end{equation}

In order to complete the proof of the first assertion, we also need to know the behaviour of the determinant of the Dirichlet Laplacian $\Delta_\beta\!\!\restriction_{|w_\pm|\leq \epsilon}$ as $\epsilon\to 0^+$. We rely on the explicit formula found in~\cite{Spreafico}. It gives
\begin{equation}\label{SPR1}
\begin{aligned}
\log\det (\Delta_\beta\!\!\restriction_{|w_\pm|\leq \epsilon})=-\frac 1 6 \left(\nu+\frac 1 \nu\right)\log \ell -\nu+1 +\frac 1 6 \nu \log 2 -(7-2\log 2)\frac 1 {12\nu}
\\
-2i\int_0^\infty\log\frac {\Gamma\bigl(\nu(1+iy)\bigr)}{\Gamma\bigl(\nu (1-iy)\bigr)}\frac {dy}{e^{2\pi y}-1}-\log\Gamma\left(\nu+1\right)+\frac 2\nu \zeta'_H(-1, \nu+1),
\end{aligned}
\end{equation}
 where $\nu=(\beta+1)^{-1}$, $\ell=\sqrt{2} (\beta+1)^{-1/2} \epsilon^{\beta+1}$, and $\zeta_H$ stands for the Hurwitz zeta function;  see the last formula in~\cite{Spreafico}. 
As it was  noticed in~\cite[Appendix B]{Klevtsov}, the integral representation
$$\begin{aligned}
\zeta'_B(0;a,b,x)&=\left(-\frac 1 2 \zeta_H\left(0,\frac x a\right)+\frac a b \zeta_H\left(-1,\frac x a \right)-\frac 1 {12}\frac b a\right)\log a +\frac 1 2 \log\Gamma\left(\frac x a\right)
\\
&-\frac 1 4\log(2\pi)-\frac a b \zeta_H\left(-1,\frac x a\right)-\frac a b \zeta'_H\left(-1,\frac x a \right)+i\int_0^\infty\log\frac {\Gamma\left(\frac {x+iby}{a}\right)}{\Gamma\left(\frac {x-iby}{a}\right)}\frac {dy}{e^{2\pi y}-1}
\end{aligned}
$$
from~\cite[Proposition 5.1]{Spreafico2} allows to express the right hand side of~\eqref{SPR1} in terms of the Barnes double zeta function~\eqref{DoubleSeries}. Thus we arrive at the equality
$$
\begin{aligned}
\log\det (\Delta_\beta\!\!\restriction_{|w_\pm|\leq \epsilon})=&-\frac {\beta^2+2\beta+2} 6 \log\epsilon+\frac {\beta^2+2\beta+2} {6(\beta+1)}\log 2 
\\
&-2\zeta_B'(0;\beta+1,1,1)-\frac 5 {12}(\beta+1) -\frac 1 2 \log(\beta+1)-\frac 1 2 \log (2\pi). 
\end{aligned}
$$
 This together with~\eqref{sst} completes the proof of the first assertion. 
 
For the proof of the second assertion we use exactly the same methods and obtain  
 \begin{equation*}\begin{aligned}
\log\det (\Delta\!\!\restriction_{|w_0|\leq \epsilon})=-\frac 1 3\log\epsilon & -\frac 1 3\frac 1{2\beta+1}\log c_\beta
\\
 &-2\zeta'_R(-1) -\frac 5 {12}  -\frac 1 2 \log (2\pi) +o(1) \quad\text{as } \epsilon\to 0^+, 
\end{aligned}
\end{equation*}
$$
\begin{aligned}
\log\det (\Delta_\beta\!\!\restriction_{|w_0|\leq \epsilon})=-\frac 1 6 \left(4\beta^2+4\beta+2\right)\log\epsilon-\frac 1 6 \left(2\beta+1+\frac 1 {2\beta+1}\right)\log 2 
\\
-2\zeta_B'(0;-2\beta-1,1,1)+\frac 5 {12}(2\beta+1) -\frac 1 2 \log(-2\beta-1)-\frac 1 2 \log (2\pi).
\end{aligned}
$$
This implies the second assertion and completes the proof of Lemma~\ref{L2}.
\end{proof}

\begin{proof}[Proof of Proposition~\ref{thm1}]  Here we glue the results of Lemma~\ref{L1} and Lemma~\ref{L2} together by using the Meyer-Vietoris type formula for determinants of Laplacians~\cite{BFK}, widely known as the BFK (Burghelea-Friedlander-Kappeler) formula, see also~\cite{Lee1,Lee2}. For the selfadjoint Laplacian $\Delta$ on the unit sphere  $(\overline{\Bbb C},m)$ the formula reads
\begin{equation}\label{BFK1}
{\det \Delta}=4\pi\det ({\Delta}\!\!\restriction_{\overline{ \Bbb C}_\epsilon})\cdot\det ({\Delta}\!\!\restriction_{|w_+|\leq \epsilon}) \cdot\det (\Delta\!\!\restriction_{|w_0|\leq \epsilon}) \cdot \det ({\Delta}\!\!\restriction_{|w_-|\leq \epsilon})\cdot \frac{\det \mathcal N\!\!\restriction_{\partial\Bbb C_\epsilon}}{\ell(\partial \Bbb C_\epsilon, m)},
\end{equation}
where $4\pi$ is the total area of the unit sphere, $\ell(\partial\Bbb C_\epsilon, m)$ stands for the length of  the boundary $\partial\Bbb C_\epsilon$ in the spherical metric $m$, and $\mathcal N\!\!\restriction_{\partial\Bbb C_\epsilon}$ is the Neumann jump operator on $\partial\Bbb C_\epsilon$ (a first order classical pseudodifferential operator). All other determinants in~\eqref{BFK1} are exactly the same as in Lemma~\ref{L1} and Lemma~\ref{L2}. 

As is well-known, $$
{\det \Delta}=\exp\bigl(1/2-4\zeta_R'(-1)\bigr),
$$
 see e.g.~\cite{OPS1},  and the quotient $\det \mathcal N\!\!\restriction_{\partial\Bbb C_\epsilon}\!\!/\ell(\partial \Bbb C_\epsilon, m)$  in~\eqref{BFK1} is conformally invariant ( as it was noticed in~\cite{Wentworth}, this can be most easily seen  from~\eqref{BFK1} together with Polyakov and Polyakov-Alvarez formulas for the determinants of Laplacians in it; see also~\cite{EW,GG}
 ). 
 
  It is also well-known that the BFK formula  remains valid for the flat singular metrics  if one picks the Friederichs selfadjoint extensions of the corresponding Laplacians (and  there are no singularities on the boundaries): the formula and its deduction hold true without any changes thanks to the same structure of short time heat trace asymptotics. A non-exhaustive list of references where the BFK formula occurs in the context of singular metrics is~\cite{HKK1,HKK2,Kalvin Pol-Alv,Kokot,Kokot2019,Kokotkin,LMP}.  For the Friederichs selfadjoint extension $\Delta_\beta$  the formula reads
$$
{\det \Delta_\beta}=\det ({\Delta_\beta}\!\!\restriction_{\overline{ \Bbb C}_\epsilon})\cdot\det ({\Delta_\beta}\!\!\restriction_{|w_+|\leq \epsilon}) \cdot\det (\Delta_\beta\!\!\restriction_{|w_0|\leq \epsilon}) \cdot \det ({\Delta_\beta}\!\!\restriction_{|w_-|\leq \epsilon})\cdot \frac{\det \mathcal N_\beta\!\!\restriction_{\partial\Bbb C_\epsilon}}{\ell(\partial \Bbb C_\epsilon,m_\beta)},
$$
where the area of envelope does not appear because it  was normalized to $1$ and the value of the quotient $\det \mathcal N_\beta\!\!\restriction_{\partial\Bbb C_\epsilon}\!\!/\ell(\partial \Bbb C_\epsilon, m)$  is the same as the value of the one in~\eqref{BFK1}. (Let us also note that the generalizations of BFK formula to the case of non-Friederichs selfadjoint extensions or non-flat singular metics are way more involved, e.g.~\cite{Kalvin Pol-Alv,LMP}.)  Summing up we come to the equality
$$\begin{aligned}
{{\det \Delta_\beta}}
=\frac {1}{4\pi}& \exp\bigl(1/2-4\zeta_R'(-1)\bigr)
\\
&\times\lim_{\epsilon\to0+}\left(\frac{ \det (\Delta_\beta\!\!\restriction_{\overline{ \Bbb C}_\epsilon})}{\det ({\Delta}\!\!\restriction_{\overline{ \Bbb C}_\epsilon})}\cdot\frac{ \det (\Delta_\beta\!\!\restriction_{|w_-|\leq \epsilon})}{\det ({\Delta}\!\!\restriction_{|w_-|\leq \epsilon})} \cdot\frac{ \det (\Delta_\beta\!\!\restriction_{|w_+|\leq \epsilon})}{\det ({\Delta}\!\!\restriction_{|w_+|\leq \epsilon})}\cdot\frac{ \det (\Delta_\beta\!\!\restriction_{|w_0|\leq \epsilon})}{\det ({\Delta}\!\!\restriction_{|w_0|\leq \epsilon})}\right).
\end{aligned}
$$
In the limit this together with Lemma~\ref{L1} and Lemma~\ref{L2} implies the stated formula~\eqref{logdet} for ${{\log \det \Delta_\beta}}$ and completes the proof of Proposition~\ref{thm1}. 
\end{proof}
\section{Absolute minimum and small-angle asymptotics}\label{Sec3}

\begin{proof}[Proof of Theorem~\ref{main}] 1. The first assertion is a direct consequence of the explicit formula for  $\log\Delta_\beta$  deduced in Proposition~\ref{thm1}, where all terms in the right hand side are real analytic on the open interval $(-1,-1/2)$; for the properties of Barnes double zeta function see e.g.~\cite{Matsumoto,Spreafico2}.

2. The small-angle asymptotic expansions~\eqref{As1} and~\eqref{As2}  readily follow from the explicit formula for  $\log\det\Delta_\beta$  found in Proposition~\ref{thm1}, the well-known asymptotics~\eqref{BAS} for  the derivative of Barnes double zeta function, and the  asymptotics 
$$
\begin{aligned}
\log c_\beta&=\frac 1 2\log(\beta+1)+\frac 1 2 \log 2 -\frac 1 2 \log\pi  -2(\beta+1)\log 2 +O\left( (\beta+1)^2\right)\ \text{as}\ \beta\to-1^+,
\\
\log c_\beta&= \frac 1 2 \log \left(-2\beta-1 \right) -\frac 1 2 \log\pi  + (2\beta+1)\log 2+O\left((2\beta+1)^2\right)\ \text{as}\ \beta\to-1/2^-
\end{aligned}
$$
for the scaling factor~\eqref{area}.

3. Let us first show that $\beta=-2/3$ is a critical point of the right hand side in~\eqref{logdet}. From Lemma~\ref{14:34} that we prove in Appendix~\ref{APX} it follows that 
$$
\frac{d}{d\beta}\bigl\{-4\zeta_B'(0;\beta+1,1,1)-2\zeta_B'(0;-2\beta-1,1,1)\bigr\}=0\quad\text{for}\quad\beta=-2/3.
$$
This together with the formulas~\eqref{logdet} for $\log\det\Delta_\beta$ and~\eqref{area} for the scaling factor $ c_\beta$ gives
$$
\frac{d}{d\beta}\log\det\Delta_\beta\Bigr|_{\beta=-2/3}= \frac 1 6 \left( -16\log 2-8\left(-\Psi\left(\frac 1 3\right) +\Psi\left(\frac 1 6\right )   +\pi\cot\frac \pi 3 \right) \right),
$$
where $\Psi$ is the digamma function and
$$
-\Psi\left(\frac 1 3\right) +\Psi\left(\frac 1 6\right )   +\pi\cot\frac \pi 3=-2\log 2
$$
by the Gauss's digamma theorem.  This demonstrates that 
$$
\frac{d}{d\beta}\log\det\Delta_\beta\Bigr|_{\beta=-2/3}=0,
$$
so $\beta=-2/3$ is a critical point of the function $\beta\mapsto\log\det\Delta_\beta$. 

For the rational values of $a$ the derivative $\zeta'_B(0;a,1,1)$ of the Barnes double zeta can be expressed in terms of $\zeta_R'(-1)$ and gamma functions, see e.g.~\cite{Dowker} and Remark~\ref{Barnes} in Appendix~\ref{APX}. In particular, this allows to find the critical value, cf.~\eqref{CritVal}. 

We have demonstrated that the equilateral triangle envelope (the most symmetrical geometry) gives rise to a critical point of the determinant on the isosceles triangle envelopes of unit area. In the remaining part of the proof we show that it provides the determinant with the absolute minimum and there are no other critical points.

It suffices to show that  the second derivative of the function 
\begin{equation*}
(-1,-1/2)\ni\beta\mapsto\log\det\Delta_\beta
\end{equation*}
is strictly positive. With this aim in mind 
we intend to  approximate  the second derivative of the right hand side in~\eqref{logdet}  by an elementary function.

   We start with the term involving the scaling factor $c_\beta$. Due to~\eqref{area} we have 
\begin{equation}\label{logSa}
-2\log  c_\beta=2\log\Gamma(\beta+1)+2\log \Gamma(-\beta-1/2)+ \log \sin\bigl(\pi(-2\beta-1)\bigr)-\log(4\pi).
\end{equation}
Expanding $\log\Gamma(\beta+2)=\log \Gamma(\beta+1)+\log(\beta+1)$ into the Taylor series we obtain
$$
\log \Gamma(\beta+1)=-\log(\beta+1)  -\gamma (\beta+1)+\sum_{k=2}^\infty \frac{\zeta_R(k)}{k} (-\beta-1)^k,\quad  |\beta+1|<1,
$$
where we can replace $\beta+1$ by $-\beta-1/2$ to get a similar representation for $\log \Gamma(-\beta-1/2)$. 
As a result~\eqref{logSa} takes the form  
$$
\begin{aligned}
-2\log  c_\beta= \log\frac { \sin\bigl(\pi(-2\beta-1)\bigr)}{\pi(\beta+1)^2(2\beta+1)^2}-\gamma+2\sum_{k=2}^\infty \frac{\zeta_R(k)}{k} \Bigl((-\beta-1)^k+ (\beta+1/2)^k\Bigr).
\end{aligned}
$$
Thus for the second derivative of the term involving  $  c_\beta$ we obtain
$$
\begin{aligned}
&-\frac 1 6 \frac {d^2}{d\beta^2}\left(\left(\frac 2 {\beta+1} -\frac 1 {2\beta+1 } -1\right)\log  c_\beta\right)> \frac {d^2}{d\beta^2}\left[\frac 1 {12}\left(\frac 2 {\beta+1} -\frac 1 {2\beta+1 } -1\right)\right.
\\ 
&\qquad\qquad\times\Biggl (  \log\frac { \sin\bigl(\pi(-2\beta-1)\bigr)}{\pi(\beta+1)^2(2\beta+1)^2} -\gamma\left.+2\sum_{k=2}^3 \frac{\zeta_R(k)}{k} \Bigl((-\beta-1)^k+ (\beta+1/2)^k\Bigr)\Biggr) \right]
\\
 \end{aligned}
$$
where the Leibniz's test was implemented 
to estimate the sums of  infinite series on the interval $-1\leq \beta\leq -1/2$. 

Next we consider the terms involving the Barnes double zeta functions. We find suitable approximations for Barnes double zeta function derivatives in Lemma~\ref{BUE}, see Appendix~\ref{APX}.   The inequality~\eqref{ZBE2} from Lemma~\ref{BUE}, where we take $N=2$,  implies
$$
\begin{aligned}
&\frac {d^2}{d\beta^2}\Bigl(-4\zeta'_B(0;\beta+1,1,1) -2\zeta'_B(0;-2\beta-1,1,1)\Bigr)
\\
&\qquad+\left(\frac1 {12} -\zeta'_R(-1)\right)\left(\frac 8 {(\beta+1)^3}-\frac 2 {(\beta+1/2)^3}\right)\geq -\frac {\zeta_R(3)(-3\beta-1)} {15} > -\frac 1 5.
\end{aligned}
$$
In total we get 
\begin{equation}\label{EF}
\begin{aligned}
&\frac {d^2}{d\beta^2} \log\det\Delta_\beta >  \frac {d^2}{d\beta^2}\left[\frac 1{12}\left(\frac 2 {\beta+1} -\frac 1 {2\beta+1 } -1\right) \right.
\\
&\times\Biggl ( \log\frac { \sin\bigl(\pi(-2\beta-1)\bigr)}{\pi(\beta+1)^2(2\beta+1)^2} -\gamma+2\sum_{k=2}^3 \frac{\zeta_R(k)}{k} \Bigl((-\beta-1)^k+ (\beta+1/2)^k\Bigr)\Biggr)
\\
&\left. +\frac 1 6  \left(\frac 4 {\beta+1}-\frac 1{2\beta+1}\right)\log 2+\frac 1 2 \log\frac 1{(\beta+1)^2(-2\beta-1)}\right]
\\& \qquad \qquad\qquad \qquad \qquad\qquad  -\frac 1 5 -\left(\frac1 {12} -\zeta'_R(-1)\right)\left(\frac 8 {(\beta+1)^3}-\frac {2} {(\beta+1/2)^3}\right).
\\
\end{aligned}
\end{equation}
One can check that the elementary function in the right hand side of the latter inequality is strictly positive on the interval $-1<\beta<-1/2$. This is straightforward and  we omit the details; see Fig.~\ref{EFG} for a graph of the elementary function.

\begin{figure}[h]
\centering\includegraphics[scale=.6]{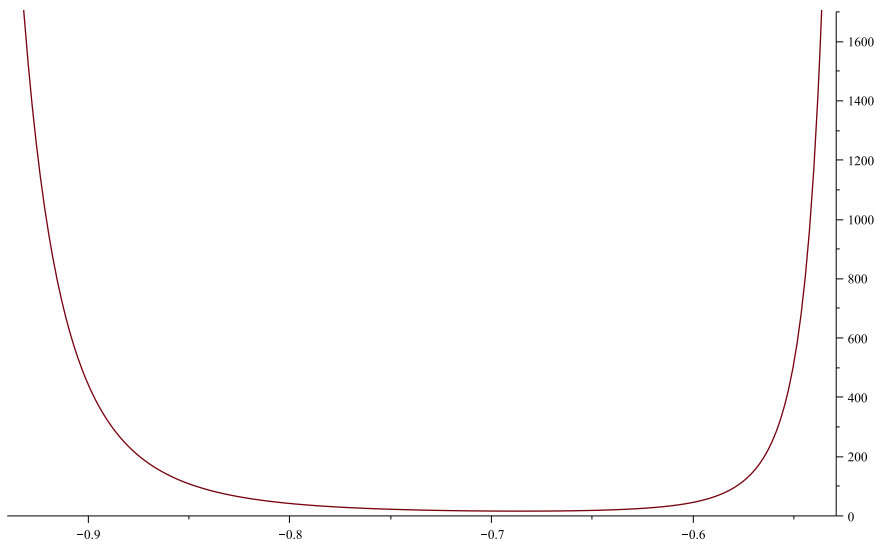}
\caption{Graph of the elementary function in the right hand side of~\eqref{EF}.}
\label{EFG}
\end{figure}

We have demonstrated that the function $\beta\mapsto\log\det\Delta_\beta$ is concave up. This completes the proof of Theorem~\ref{main}.
\end{proof}
 
\section{Determinant on  envelopes of non-unit area}\label{Sec 4}
The Laplacian $\Delta_\beta^{S}$ on the Euclidean isosceles triangle envelopes of  area $S\neq S(\beta)$ is  generated by the metric $m_\beta^S=S\cdot m_\beta$ on $\overline{\Bbb C}$ , where $m_\beta$ is the same as in~\eqref{metric}. 
Since for the Friederichs Laplacians we have $\Delta_\beta^{S}=\frac 1 S\Delta_\beta$, differentiating  the spectral zeta function $\zeta^S_\beta(s)=\sum_{j>0} (\lambda_j/S)^{-s}$ of $\Delta_\beta^{S}$  with respect to $s$ we arrive at the standard rescaling property
\begin{equation}\label{resc1}
\log\det\Delta_\beta^{S}=\log\det\Delta_\beta-\zeta_\beta(0)\log S.
\end{equation}
A more serious task is to find that 
\begin{equation}\label{resc2}
\zeta_\beta(0)=-\frac {13} {12} +\frac 1 {6(\beta+1)}    -\frac 1 {12(2\beta+1)}.
\end{equation}
This can be done either by repeating the steps in the proof of Proposition~\ref{thm1} (now for the metric $m_\beta^S$)  or by relying on a more general result~\cite[Section 1.2]{Kalvin Pol-Alv}.

\begin{remark}\begin{enumerate}\label{Area A}
\item Theorem~\ref{main}.1 together with~\eqref{resc1} and~\eqref{resc2} implies that
 \begin{equation}\label{DetEnvS}
 (-1,-1/2)\ni\beta\mapsto  \log\det\Delta_\beta^{S}
 \end{equation}
 is a real analytic function (for the isosceles triangle envelopes of any  fixed area $S$). 
\item With the help of~\eqref{resc1},~\eqref{resc2} the small-angle asymptotics  for $\log\det\Delta_\beta^{S}$ can be easily obtained from those in Theorem~\ref{main}.2. In particular, it turns out that the determinant  $\det \Delta^S_\beta$ grows without any bound as the isosceles triangle envelope of area $S$ degenerates (i.e.  as an internal angle of triangle $ABC$  of area $S/2$ goes to zero, or, equivalently,  as $\beta\to -1^+$ or $\beta\to-1/2^-$).

\item For  $\zeta_\beta(0)$ in~\eqref{resc2} we have $\frac d {d\beta}\zeta_\beta(0)\bigr|_{\beta=-2/3}=0$ and 
$\zeta_{-2/3}(0)=-\frac 1 3$.
Therefore as an immediate consequence of Theorem~\ref{main}.3 we conclude that $\beta=-2/3$  is a critical point of the function $\beta\mapsto  \log\det\Delta_\beta^{S}$ for any $S>0$.   Thanks to~\eqref{resc1}
for the   critical value we have
$$
\log\det\Delta_{-2/3}^{S}=\frac 2 3 \log \pi  +\frac 1 3 \log \frac 2 3      -2\log\Gamma \left(\frac 2 3 \right) +\frac 1 3 \log S.
$$
Recall that $\beta=-2/3$ corresponds to the equilateral triangle envelope, i.e. to the most symmetrical geometry.

\item Since  $\frac {d^2} {d\beta^2}\zeta_\beta(0)>0$,   it is also true that for any $S\leq 1$ the function~\eqref{DetEnvS} is concave up, has exactly one critical point,  and 
$$
\log\det\Delta_{\beta}^{S}\geq \log\det\Delta_{-2/3}^{S} \quad \text{for}\quad -1<\beta<-1/2
$$
with equality iff $\beta=-2/3$; cf. Theorem~\ref{main}.3. In other words, the determinant of Friederichs Laplacian on the isosceles triangle envelopes of fixed area $S\leq1$ reaches its \underline{absolute minimum}  on the equilateral triangle envelope.  

\item In fact, the estimate~\eqref{EF} together with~\eqref{resc1} and~\eqref{resc2} is also capable of showing that $\beta=-2/3$ still corresponds to the absolute minimum of $\log\det\Delta_{\beta}^{S}$ for the areas $S$ slightly greater than one. However, with further increase of $S$ the value of
$$
\frac {d^2}{d \beta^2}\zeta_\beta(0)\Bigr|_{\beta=-2/3}\log S=27 \log S
$$
becomes greater than 
$$
\frac {d^2}{d\beta^2}\log \det\Delta_\beta\bigr|_{\beta=-2/3}
\approx17.614
$$ 
(approximately for $S>1.92$). As a result, the second derivative  $$\frac {d^2}{d\beta^2} \log\det\Delta^S_\beta\Bigr|_{\beta=-2/3}$$ becomes negative, cf.~\eqref{resc1},  and  the critical point $\beta=-2/3$  turns  into a local maximum  of $\log\det\Delta_\beta^{S}$, see~Fig.~\ref{EnvCritArea}. 

\begin{figure}[h]
\centering\includegraphics[scale=.39]{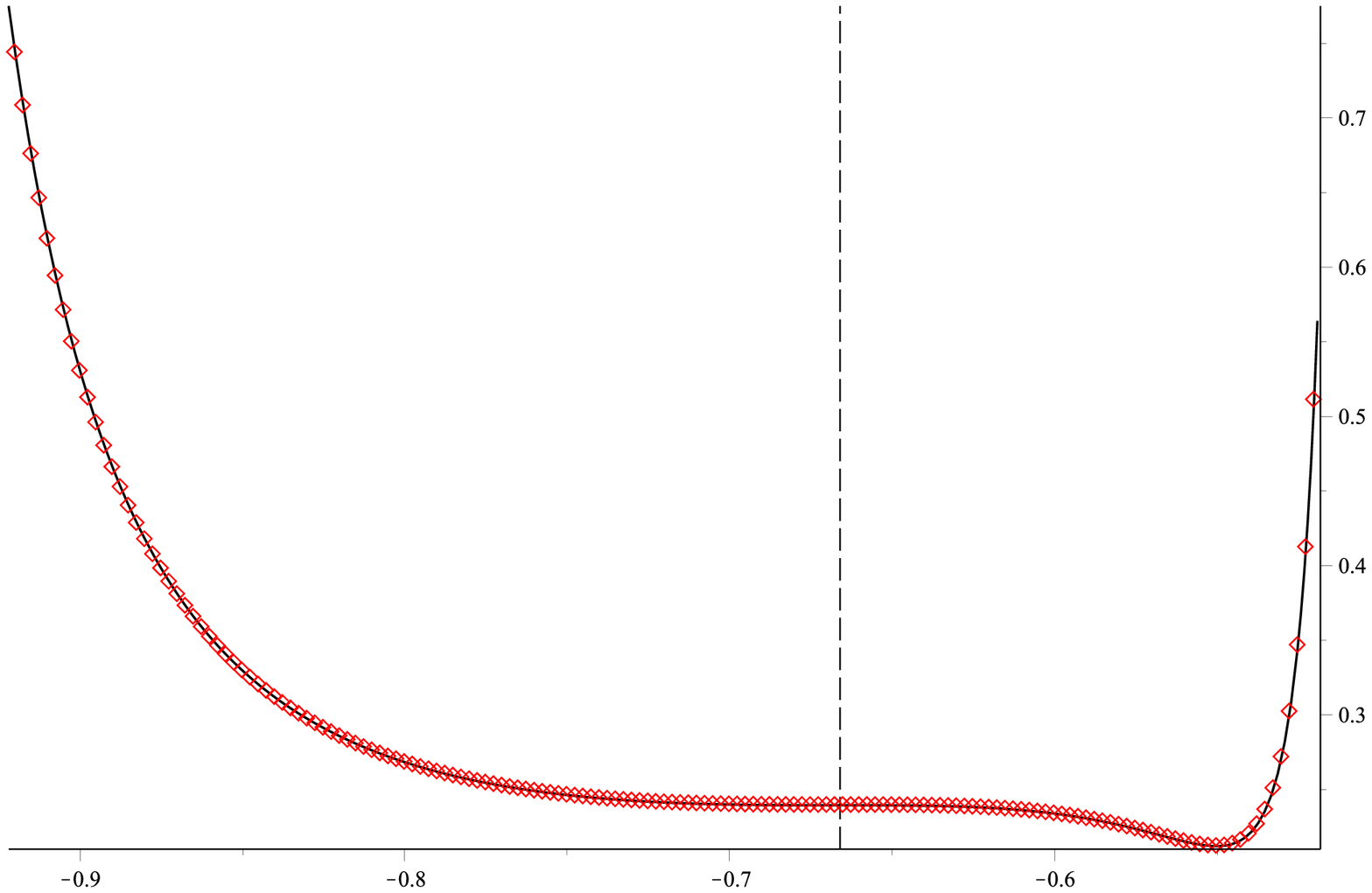}\hspace{5pt}\includegraphics[scale=.39]{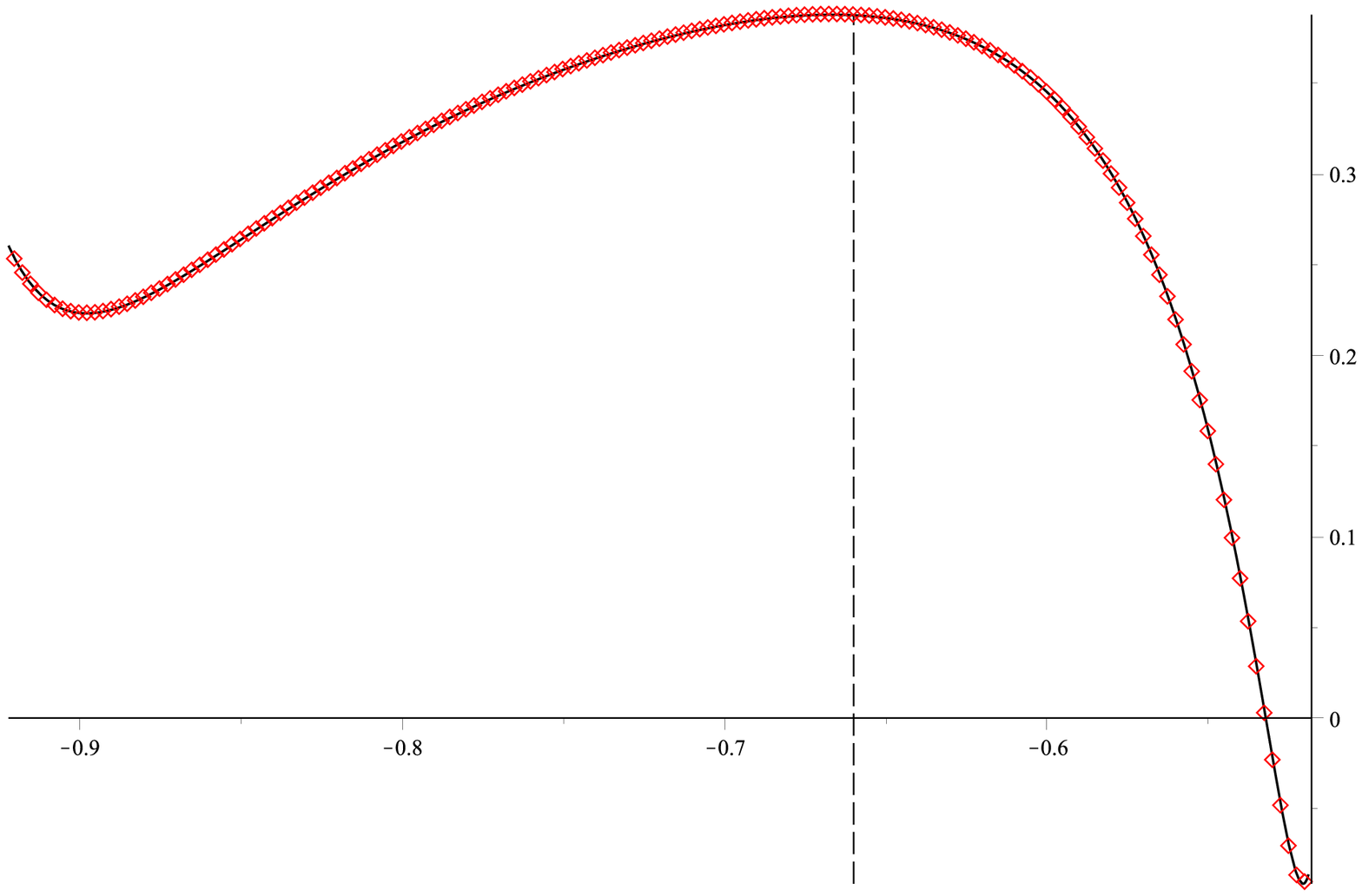}
\caption{Graph of the  function $(-1,-1/2)\ni\beta\mapsto \log\det\Delta^S_\beta$ for  the Euclidean isosceles triangle envelopes of total area $S=1.92$ (on the left) and $S=3$ (on the right). The dashed line corresponds to the critical point $\beta=-2/3$ or, equivalently, to the equilateral triangle envelope (the most symmetrical geometry). For the graphs we use the formula for $ \log\det \Delta_\beta$  found in Proposition~\ref{thm1} together with rescaling formulas~\eqref{resc1},~\eqref{resc2}. We  find  the exact values of $ \log\det \Delta^S_\beta$ for some rational values of $\beta$ by using Remark~\ref{Barnes} (the points marked with Diamond symbol). We also  uniformly approximate $ \log\det \Delta^S_\beta$ by using the estimate~\eqref{ZBE1} in Lemma~\ref{BAS}, where we take $N=4$ (solid line).}
\label{EnvCritArea}
\end{figure}

A similar effect also appears on  a sphere with positive constant curvature (spherical) metrics of normalized area and two antipodal singularities, see~\cite[Section 3.1]{Kalvin Pol-Alv}.
\end{enumerate}
\end{remark}

\appendix

\section{Derivatives of Barnes double zeta function}\label{APX}

The asymptotics 
\begin{equation}\label{BAS}
\begin{aligned}
\zeta'_B(0;a,1,1)=\left(\frac 1 {12} -\zeta'_R(-1)\right)\frac 1 a& -\frac 1 4 \log (2\pi)+\frac {\gamma a}{12} 
\\
&+\sum_{k=2}^\infty\frac {B_{2k}\zeta_R(2k-1)}{2k(2k-1)}a^{2k-1},\quad a\to 0^+,
\end{aligned}
\end{equation}
for the derivative of the Barnes double zeta function~\eqref{DoubleSeries} with respect to its first argument is well known, see e.g.~\cite{Matsumoto,Spreafico2}.  In this paper we use  an improvement of this result, which we formulate and prove in Lemma~\ref{BUE} below.  
\begin{lemma}[Approximations for derivatives of Barnes double zeta function]\label{BUE} For $a>0$  and any $N\geq2$ we have
\begin{equation}\label{ZBE1}
\begin{aligned}
\left |\zeta'_B(0;a,1,1)-\left(\frac 1 {12} -\zeta'_R(-1)\right)\frac 1 a +\frac 1 4 \log (2\pi)-\frac {\gamma a}{12} -\sum_{k=2}^{N-1}\frac {B_{2k}\zeta_R(2k-1)}{2k(2k-1)}a^{2k-1}\right|
\\
\leq  \frac {|B_{2N}\zeta_R(2N-1)|}{2N(2N-1)}a^{2N-1},
\end{aligned}
\end{equation}
\begin{equation}\label{ZBE2}
\begin{aligned}
\left |\frac{\partial^2}{\partial a^2}\zeta'_B(0;a,1,1)-2\left(\frac 1 {12} -\zeta'_R(-1)\right)\frac 1 {a^3} -\sum_{k=2}^{N-1}\left(1-\frac 1 k\right)B_{2k}\zeta_R(2k-1)a^{2k-3}\right|
\\
\leq  \left(1-\frac 1 N\right)|B_{2N}\zeta_R(2N-1)|a^{2N-3};
\end{aligned}
\end{equation}
for $N=2$ the sums with respect to $k$ do not appear. 
Here $\gamma=-\Gamma'(1)$ is the Euler's constant,  $B_{n}$ is the $n$-th  Bernoulli number,  $\zeta_R$ is the Riemann zeta function, and the prime denotes the derivative of $\zeta_B(s;a,b,x)$ with respect to $s$.
\end{lemma}
Recall that  the proof of Theorem~\ref{main}.3 relies on the estimate~\eqref{ZBE2}. The estimate~\eqref{ZBE1} with $N=4$ was only used in order to plot graphs of LogDet function on Fig.~\ref{F1}  and  Fig.~\ref{EnvCritArea}  (solid line); for $N=4$ the right hand side of~\eqref{ZBE1} does not exceed  $5\times 10^{-3} \times a^7$, which gives more than an adequate approximation of $\zeta'_B$ for  the graphs.
 \begin{proof}
We  rely on the  integral representation\footnote{The representation of $\zeta'_B(0;a,1,1)$ in terms of $J(a)$ was  first found in~\cite{Kalvin Pol-Alv}. Other known integral representations of $\zeta'_B(0;a,1,1)$, see e.g.~\cite{Matsumoto,Spreafico2,Noronha} and references therein, are not  good  for our purposes as it is hard to differentiate them with respect to the parameter $a$. To verify the validity of a representation for $\zeta'_B(0;a,1,1)$ it suffices  to check it only for the rational values of $a$ (due to the analytic regularity of $\zeta'_B(0;a,1,1)$ in the half plane $\Re a>0$).  It is known that  for rational values of $a$ the value of $\zeta'_B(0;a,1,1)$ can be expressed in terms of  Riemann zeta and gamma functions, see e.g.~\cite{Dowker} and Remark~\ref{ZB}.  The above representation returns exactly the same result indeed; a detailed evaluation of $J(a)$  for rational values of $a$ can be found in~\cite{Au-Sal}.   }
\begin{equation}\label{pst}
\zeta'_B(0;a,1,1)=\frac 1 {12}\left(a+\frac 1 a \right)(\gamma-\log a)-\frac 1 4\log a+\frac 5 {24} a-\frac 1 4 \log(2\pi)+J(a),
\end{equation}
where $\gamma=-\Gamma'(1)$ and 
$$
J(a)=\int_0^\infty\frac 1 {e^t-1}\left[\frac 1 {2t}\coth\frac t {2a}-\frac a 4 \csch^2\frac t 2 -\frac 1 {12} \left(a+\frac 1 a \right)\right]\,dt.
$$ 
We have
\begin{equation}\label{Jexp}
J(a)=\int_0^\infty\left( \frac t {e^t-1}\frac 1{t^2(e^{t/a}-1)}+ \frac 1 {e^t-1}\left[\frac 1 {2t}-\frac a 4 \csch^2\frac t 2 -\frac 1 {12} \left(a+\frac 1 a \right)\right]\right)\,dt.
\end{equation}
Expanding the factor $ \frac t {e^t-1}$ into the Taylor series and applying the  Leibniz's test we obtain
\begin{equation}\label{Jexp1}
\left| \frac t {e^t-1} -1+\frac 1 2 t-\frac 1 {12} t^2-\sum_{k=2}^{N-1} \frac {B_{2k}}{(2k)!} t^{2k}\right|\leq   \frac {|B_{2N}|}{(2N)!} t^{2N},\quad t\geq 0.
\end{equation}
It is not hard to check that 
\begin{equation}\label{Jexp2}
\begin{aligned}
\int_0^\infty\left( \left(\frac 1 {t^2} -\frac 1 {2t} +\frac 1 {12}\right)\frac 1{e^{t/a}-1}+ \frac 1 {e^t-1}\left[\frac 1 {2t}-\frac a 4 \csch^2\frac t 2 -\frac 1 {12} \left(a+\frac 1 a \right)\right]\right)\,dt
\\
=\frac 1 {12}\left(a+3+\frac 1 {a}\right)\log a+\left(\frac {1-\gamma}{12} -\zeta'_R(-1)\right)\frac 1 a -\frac 5 {24}a. 
\end{aligned}
\end{equation}
We also note that the well known identity $\zeta(n)(n-1)!=\int_0^\infty \frac {t^{n-1}\,dt}{e^t-1}$ gives
\begin{equation}\label{Jexp3}
\int_0^\infty  \frac {B_{2k}}{(2k)!} t^{2k-2}\frac 1{e^{t/a}-1}\,dt = \frac {B_{2k}}{(2k)!} a^{2k-1} \int_0^\infty  \frac {t^{2k-2}}{e^{t}-1}\,dt=  \frac {B_{2k}\zeta_R(2k-1)}{2k(2k-1)} a^{2k-1}.
\end{equation}

Now~\eqref{Jexp}--\eqref{Jexp3}  lead to the estimate 
$$
\begin{aligned}
\left|J(a)-\frac 1 {12}\left(a+3+\frac 1 {12}\right)\log a-\left(\frac {1-\gamma}{12} -\zeta'_R(-1)\right)\frac 1 a +\frac 5 {24}a\right.\qquad\qquad\qquad
\\
-\left. \sum_{k=2}^{N-1}\frac {B_{2k}\zeta_R(2k-1)}{2k(2k-1)} a^{2k-1}\right|\leq  \frac {|B_{2N}\zeta_R(2N-1)|}{2N(2N-1)}a^{2N-1}.
\end{aligned}
$$
This together with~\eqref{pst}   completes  the proof of~\eqref{ZBE1}.

In order to prove~\eqref{ZBE2} we first write
\begin{equation}\label{Aug25}
\begin{aligned}
\zeta_B'(0;a,1,1)-\left(\frac 1 {12} -\zeta'_R(-1)\right)\frac 1 a +\frac 1 4 \log (2\pi)-\frac {\gamma a}{12}
\\
=\int_0^\infty\left( \frac t {e^t-1}-1+\frac 1 2 t-\frac 1 {12} t^2 \right)\frac 1{t^2(e^{t/a}-1)}\,dt
\\
=\int_0^\infty\frac { F(at)}{t(e^{t}-1)}\,dt,
\end{aligned}
\end{equation}
where we introduced the notation
\begin{equation}\label{Aug252}
F(r)=  \frac 1 {e^{r}-1}-\frac 1{r}+\frac 1 2 -\frac 1 {12} r \left(=\sum_{k=2}^{\infty} \frac {B_{2k}}{(2k)!} r^{2k-1}\right).
\end{equation}
Since $|F'(r)/r|<1$ and $|F''(r)|<1$ we have 
$$
\int_0^\infty\left| \frac {\partial}{\partial a} F(at)\right|\frac 1{t(e^{t}-1)}\,dt=  \int_0^\infty  |F'(at)|\frac 1{e^{t}-1}\,dt< \int_0^\infty  \frac {at}{e^{t}-1}\,dt=\frac {\pi^2} 6 a
$$
$$
\int_0^\infty\left | \frac{\partial^2}{\partial a^2} F(at)\right|\frac 1{t(e^{t}-1)}\,dt=  \int_0^\infty  |F''(at)|\frac t{e^{t}-1}\,dt< \int_0^\infty  \frac {t}{e^{t}-1}\,dt=\frac {\pi^2} 6.
$$
This demonstrates that 
$$
\frac{\partial^2}{\partial a^2} \int_0^\infty F(at)\frac 1{t(e^{t}-1)}\,dt=\int_0^\infty \frac{\partial^2}{\partial a^2}F(at)\frac 1{t(e^{t}-1)}\,dt.
$$
Now as a consequence of~\eqref{Aug25} we  conclude that 
\begin{equation}\label{TZBE2}
\frac{\partial^2}{\partial a^2}\zeta'_B(0;a,1,1)-2\left(\frac 1 {12} -\zeta'_R(-1)\right)\frac 1 {a^3}=\int_0^\infty \frac{\partial^2}{\partial a^2}F(at)\frac 1{t(e^{t}-1)}\,dt.
\end{equation}
For $at\geq 0$ the second partial derivative $\frac{\partial^2}{\partial a^2}F(at)=t^2 F''(at)$ can be written as the following convergent alternating series:
$$
 \frac{\partial^2}{\partial a^2}F(at)= t^2\sum_{k=2}^{\infty} \frac {B_{2k}}{2k(2k-3)!}  (at)^{2k-3}. 
$$
The Leibniz's test gives
$$
\left| \frac{\partial^2}{\partial a^2}F(at)-  \sum_{k=2}^{N-1} \frac {B_{2k}}{2k(2k-3)!}  a^{2k-3}  t^{2k-1} \right|\leq \frac {|B_{2N}|}{2N(2N-3)!}  a^{2N-3}  t^{2N-1}.
$$
This together with~\eqref{TZBE2} leads to~\eqref{ZBE2} (in the same way as in the proof of~\eqref{ZBE1}).
\end{proof}
\begin{lemma}\label{14:34}
 For the Barnes double zeta function  $\zeta_B(s;a,b,x)$ we have 
$$
\frac{\partial}{\partial a}\{2\zeta_B'(0;a,1,1)+\zeta_B'(0;1-2a,1,1)\}=0\quad\text{for}\quad a=1/3.
$$
As before the prime stands for the derivative with respect to $s$. 
\end{lemma}
\begin{proof} As in the proof of Lemma~\ref{BUE} we conclude that 
 $$
\frac{\partial}{\partial a} \zeta_B'(0;a,1,1)= -\left(\frac 1 {12} -\zeta'_R(-1)\right)\frac 1 {a^2}+\frac {\gamma }{12}+\int_0^\infty \frac {F'(at)}{e^{t}-1}\,dt,\quad a>0,
$$
where $F(r)$ is the same as in~\eqref{Aug252}; cf.~\eqref{Aug25}. This immediately implies the assertion. 
\end{proof}

\begin{remark}[Particular values of Barnes double zeta derivative]\label{Barnes}
If $a$ is a rational number, then the value of the derivative $\zeta'_B(0;a,1,1)$  can be expressed  in the following way:
\begin{equation}\label{ZB}
 \begin{aligned}
\zeta'_B(0;p/q,1,1)=&\frac 1 {pq}\left(\zeta_R'(-1)-\frac{\log q}{12}\right)
+\left(S(q,p)+\frac 1 4 \right)\log{\frac{q}{p}}
\\
+\sum_{k=1}^{p-1}\left(\frac 1 2 -\frac k p \right)&\log \Gamma\left(  \left(\!\!\!\left(\frac{kq}{p}\right)\!\!\!\right)+\frac 1 2 \right)+\sum_{j=1}^{q-1}\left(\frac 1 2 -\frac j q \right)\log \Gamma\left(\left(\!\!\!\left(\frac{jp}{q}\right)\!\!\!\right)+\frac 1 2\right).
\end{aligned}
\end{equation}
Here  $p$ and $q$ are coprime natural numbers, $S(q,p)=\sum_{j=1}^{p}\left(\!\!\left(\frac{j}{p}\right)\!\!\right) \left(\!\!\left(\frac{jq}{p}\right)\!\!\right)$ is the Dedekind sum, and the symbol $(\!(\cdot)\!)$ is defined so that   $(\!(x)\!)=x-\lfloor x\rfloor-1/2$ for $x$ not an integer and $(\!(x)\!)=0$ for $x$ an integer; for details we refer to~\cite[Appendix A]{Kalvin Pol-Alv}. In particular, $\zeta'_B(0;1,1,1)=\zeta_R'(-1)$ and for the reciprocals of the natural numbers the general formula~\eqref{ZB} simplifies to
$$
\zeta_B'\left(0; 1 /q,1,1\right)=\frac 1 q \zeta'_R(-1)-\frac 1 {12q}\log q-\sum_{j=1}^{q-1}\frac j q\log\Gamma\left(\frac j q\right)+\frac{q-1}4 \log 2\pi.
$$
\end{remark}

In the context of this paper Remark~\ref{Barnes} becomes extremely helpful if $\beta$ is a rational number, i.e. the angles of triangles $ABC$ and $CBA'$ are rational multiples of $\pi$, cf.~Fig.~\ref{Template}. 
Thus by using~\eqref{logdet}  together with~\eqref{ZB}   for the envelope of unit area glued from two right isosceles triangles  we obtain
$$
\log\det\Delta_{-3/4}=\frac 1 4 \log \pi-\log\Gamma\left(\frac 3 4\right)\approx 0.0829.
$$
Similarly, for the equilateral triangle envelope of unit area  we get
\begin{equation}\label{CritVal}
\log\det\Delta_{-2/3}=\frac 2 3 \log \pi  +\frac 1 3 \log \frac 2 3      -2\log\Gamma \left(\frac 2 3 \right)\approx 0.0217;
\end{equation}
 for the envelope of unit area glued from two congruent  isosceles triangles with interior angles $\pi/6$, $2\pi/3$ and $\pi/6$ we have
 $$ 
 \begin{aligned}
\log\det\Delta_{-5/6}=\frac 1 6\log \pi  +\frac{37}{72}\log 2 +\frac 1{48}\log 3 +\zeta_R'(-1) -\frac 1 4 \log \Gamma\left(\frac 2 3\right)\approx0.3287.
\end{aligned}
$$
On Fig.~\ref{F1} and Fig.~\ref{EnvCritArea} all points marked by the diamond symbol were found in exactly the same way. Let us also note that  for  some exceptional values of $\beta$ it is possible to find all eigenvalues of the Friederichs Laplacian $\Delta^S_\beta$ and then the corresponding zeta regularized spectral determinant, see~\cite[Section B]{Au-Sal 2}.

\vspace{.5cm}
\noindent{\it Acknowledgements.} The author would like to thank Dr.~Alexey Kokotov (Concordia University), Dr.~Werner M\"uller (University of Bonn), and  Dr.~Paul Wiegmann (University of Chicago) for valuable discussions.

\end{document}